\documentclass{article}

\usepackage{amsfonts}
\usepackage{amsmath}
\usepackage{amssymb}
\usepackage{graphics}
\usepackage{graphicx}
\usepackage{color}
\usepackage{float}
\usepackage{epsfig}
\usepackage{geometry}
\usepackage{indentfirst,latexsym,bm}
\usepackage{lscape}
\usepackage{tabularx}
\usepackage{enumerate}
\usepackage{longtable}
\usepackage{lscape}
\numberwithin{equation}{section}

\newtheorem{theorem}{Theorem}[section]

\newtheorem{assumption}[theorem]{Assumption}

\newtheorem{lemma}[theorem]{Lemma}

\newtheorem{proposition}[theorem]{Proposition}
\newtheorem{remark}[theorem]{Remark}

\newenvironment{proof}[1][Proof]{\textbf{#1.} }{\ \rule{0.5em}{0.5em}}

\title{An approximation scheme for variational inequalities with convex and coercive Hamiltonians}
\author{Shuo Huang\thanks{Department of Statistics, The University of Warwick, Coventry CV4
7AL, U.K. \texttt{s.huang.13@warwick.ac.uk} }}

\date{}
\begin{document}
%\date{First version:\ Jan 2018; This version: \today}
\maketitle
%%%%%%%%%%%%%%%%%%%%%%%%%%%%%%%%%%%%%%%%%%%%%%%%%%%%%%%%%%%
\begin{abstract}
We propose an approximation scheme for a class of semilinear variational inequalities whose Hamiltonian is convex and coercive. The proposed scheme is a natural extension of a previous splitting scheme proposed by Liang, Zariphopoulou and the author for semilinear parabolic PDEs. 
We establish the convergence of the scheme and determine the convergence rate by obtaining its error bounds. The bounds are obtained by Krylov's shaking coefficients technique and Barles-Jakobsen's optimal switching approximation, in which a key step is to introduce a variant switching system.
\\


%\noindent\textit{Mathematics Subject Classification (2010)}: 91G40,
%\and 91G80, \and 60H30.
\end{abstract}

%%%%%%%%%%%%%%%%%%%%%%%%%%%%%%%%%%%%%%%%%%%%%%%%%%%%%%%%%%%
%%%%%%%%%%%%%%%%%%%%%%%%%%%%%%%%%%%%%%%%%%%%%%%%%%%%%%%%%%

\section{Introduction}
This paper is an extension of a previous work started by Liang, Zariphopoulou and the author \cite{Huang}, in which they consider semilinear parabolic PDEs with convex and coercive Hamiltonians, and propose an approximation based on splitting the equation into a linear parabolic equation and a Hamilton-Jacobi equation. By the convexity property of Hamiltonians, the semilinear parabolic PDEs they considered can 
be written as HJB type of parabolic equations, which correspond to stochastic optimal control problems. Herein, we use the same setting and extend their work by treating optimal stopping as well as optimal control at the same time. This leads to obstacle problems with associated variational inequalities. To be more specific, we consider semilinear parabolic variational inequalities of the form
\begin{align}
\max\{-\partial_tu+g(t,x,\partial_{x}u,\partial_{xx}u),u-f(t,x)\}&=0 & \text{in}\ Q_T;\label{PDE_1}\\
u(T,x)&=U(x) & \text{in}\ \mathbb{R}^n,\label{terminal}
\end{align}
where
$$g(t,x,p,X):=-\frac{1}{2}\text{tr}\left(\sigma\sigma^T(t,x)X\right)-b(t,x)\cdot p+H(t,x,p),$$
and $Q_T=[0,T)\times\mathbb{R}^n$. A key feature is that the
Hamiltonian $H(t,x,p)$ is convex and coercive in $p$. 
In particular, this covers the case that $H$ has quadratic growth in $p$,
a case that corresponds to a rich class of equations in mathematical
finance arising, for example, in optimal investment with homothetic
risk preferences (\cite{Hu}), exponential indifference valuation
(\cite{CSYY, HL, HL1}) and entropic risk measures (\cite{CHLZ}), just to
name a few. 
Note that if $u<f$ in $Q_{T}$, equation (\ref{PDE_1}) reduces exactly to the semilinear parabolic PDE considered in \cite{Huang}:
\begin{equation}\label{previousPDE}
-\partial_tu+g(t,x,\partial_{x}u,\partial_{xx}u)=0 \ \ \ \text{in}\ Q_T.
\end{equation}
%
%More broadly, these equations are inherently connected to
%(quadratic) backward stochastic differential equations (BSDE), a
%central area of stochastic analysis (\cite{DHB} \cite{Peng} and
%\cite{Kobylanski}). Specifically, the Hamiltonian $H(t,x,p)$
% is directly related to the BSDE's driver and,
%moreover, the solution of (\ref{PDE_1}) yields a functional-form
%representation of the processes solving the BSDE.

%General existence and uniqueness results can be found, among others
%in \cite{Kobylanski} as well as in \cite{Hu}, where BSDE techniques
%have been mainly applied. Closed-form solutions can be constructed
%only in one-dimensional cases (\cite{Zari}). Furthermore,
%approximation schemes have been developed; see \cite{Touzi2} and
%\cite{CR} for more references.

Herein, we contribute to proposing an approximation scheme for variational inequalities of the type (\ref{PDE_1}) using an approximation of (\ref{previousPDE}) introduced in \cite{Huang}. The key idea is to use in an
essential way the \textit{convexity} of the Hamiltonian. To the best of our knowledge, this property has not been adequately exploited in the existing
approximation studies. The extension from the scheme for (\ref{previousPDE}) to our scheme for (\ref{PDE_1}) is natural. Suppose $S$ denotes the approximation of (\ref{previousPDE}), we propose an approximation scheme for (\ref{PDE_1}):
$$\max\left\{S(\Delta, t,x,u^\Delta(t,x),u^\Delta(t+\Delta,\cdot)), u^\Delta(t,x)-f(t,x)\right\}=0,$$
where $\Delta$ is the discretisation time step and $u^{\Delta}$ is the solution of the scheme used to approximate the (viscosity) solution of (\ref{PDE_1}). We will formally introduce the approximation scheme in section \ref{sec-model}; see (\ref{semigroupequation1}) and
(\ref{semischeme}) for details. We also refer to \cite{Huang} for the construction and intuition of the approximation $S$ of (\ref{previousPDE}).

Next, we establish the convergence of the scheme solution to the unique (viscosity)
solution of (\ref{PDE_1})-(\ref{terminal}) and determine the rate of convergence. We
do this by obtaining upper and lower bounds on the approximation
error (Theorems \ref{theorem_error_1} and \ref{theorem_error_2},
respectively). The main tools  come from the \emph{shaking
coefficients technique} introduced by Krylov \cite{Krylov1}
\cite{Krylov} and the \emph{optimal switching approximation}
introduced by Barles and Jakobsen \cite{BJ0} \cite{BJ}.

While various arguments follow from adaptations of these techniques,
a main difficulty is to derive the consistency error estimate. Fortunately, thanks to the previous work of \cite{Huang} (Proposition 2.5, 2.8 therein), the consistency estimate follows immediately herein. Using this estimate and the comparison result for the approximation
scheme (Proposition \ref{schemecomparison}), we in turn derive an
upper bound for the approximation error by perturbing the
coefficients of (\ref{PDE_1}). 

The lower bound for the approximation
error is obtained by another layer of approximation of (\ref{PDE_1}) via an auxiliary optimal switching system. Barles and Jakobsen \cite{BJ0} \cite{BJ} in their paper use standard optimal switching systems to approximate standard HJB type of equations. However, the variational inequality (\ref{PDE_1}) considered herein can not be written as a standard HJB equation but an HJB equation with obstacle. Thus, we modify the standard optimal switching system and introduce a {variant type of switching system}, which can be proved to approximate the HJB equation with obstacle. To the best of our knowledge, we are the first to introduce this \emph{variant switching system} in the existing literature. We give the well-posedness, regularity as well as continuous dependence result for this \emph{variant switching system} in Section \ref{App2}.

A highly related work to this paper is Jakobsen \cite{Jakobsen}, where he obtained error bounds for general monotone approximation schemes for Bellman equations arising in a stochastic optimal stopping and control problem. Due to the convex and coercive property of the Hamiltonian $H(t,x,p)$, (\ref{PDE_1}) can be written as the same type of Bellman equation, but with control set and coefficients unbounded. This is not the case in \cite{Jakobsen}. Furthermore, he used the same \emph{shaking
coefficients technique} to derive an error bound for one side, but for the other side he interchanged the roles of the approximation scheme and the original equation, based on an additional assumption (Assumption 2.5 therein) that the scheme solution has enough regularity. Unfortunately, our proposed scheme does not satisfy this assumption and thus, we need another layer of approximation of the original equation, which decreases the convergence rate. Finally, a common point between his work and ours is that the scheme solution $u^{\Delta}$ is defined at every point in $\bar{Q}_{T}$ rather than some certain time (and space) grids.

%Finally, we mention that most of the existing algorithms (see, among
%others, Howard's finite difference scheme \cite{BMZ}) provide
%approximations only at certain time grids. In contrast, the
%approximation scheme we propose can be used to approximate the
%solution at any time point. Furthermore, the commonly used time
%discretization algorithms for (\ref{PDE_1}) require that the
%Hamiltonian has the form $H(t,x,\sigma^{tr}(t,x)\partial_xu)$ (see
%\cite{Touzi2} and \cite{CR}), which is not the case herein. Indeed,
%we do not require the last variable in the Hamiltonian $H$ to depend
%on the diffusion coefficient $\sigma$.

The paper is organized as follows. In the next section we introduce the
approximation scheme for (\ref{PDE_1})-(\ref{terminal}). In Section 3, we prove its convergence and establish the convergence rate
by obtaining the upper and lower bounds for the approximation error, which are the main results of this paper. Section \ref{App2} is devoted to introducing the variant switching systems. We conclude in Section 5. Some technical proofs are
provided in the appendix.
%%%%%%%%%%%%%%%%%%%%%%%%%%%%%%%%%%%%%%%%%%%%%%%%%%%%%%%%5
%%%%%%%%%%%%%%%%%%%%%%%%%%%%%%%%%%%%%%%%%%%%%%%%%%%%%%%%%

\section{The approximation scheme for (\ref{PDE_1})} \label{sec-model}
Let $d\in\mathbb{Z}^{+}$ and $\delta>0$. For a function $f:Q_T\to
\mathbb{R}^{d}$, we introduce its (semi)norms
$$|f|_{0}:=\sup_{(t,x)\in Q_T}|f(t,x)|,$$
$$[f]_{1,\delta}:=\sup_{\substack{ (t,x),(t',x)\in Q_T \\t\neq t'}}\frac{|f(t,x)-f(t',x)|}{|t-t'|^{\delta}},\ \ \
[f]_{2,\delta}:=\sup_{\substack{ (t,x),(t,x')\in Q_T \\ x\neq x'
}}\frac{|f(t,x)-f(t,x')|}{|x-x'|^{\delta}}.$$ Furthermore,
$[f]_{\delta}:=[f]_{1,\delta/2}+[f]_{2,\delta}$ and
$|f|_{\delta}:=|f|_{0}+[f]_{\delta}$. Similarly, the (semi)norms of
a function $g:\mathbb{R}^n\to\mathbb{R}^{d}$ are defined as
$$|g|_{0}:=\sup_{x\in \mathbb{R}^n}|g(x)|,\ \ \ [g]_{\delta}:=\sup_{\substack{ x,x'\in \mathbb{R}^n \\ x\neq x' }}\frac{|g(x)-g(x')|}{|x-x'|^{\delta}},\ \ \ |g|_{\delta}:=|g|_{0}+[g]_{\delta}.$$
%Note that all the above norms satisfy the usual triangle inequality, and that
%\begin{equation}\label{norm}
%|f|_{\delta}=\sup_{t\in I}|f(t,\cdot)|_{\delta}+[f]_{1,\delta/2}
%\end{equation}

For $S={Q}_T$, $\mathbb{R}^n$ or {$Q_{T}\times \mathbb{R}^{n}$}, we
denote by $\mathcal{C}(S)$ the space of  continuous real-valued functions on $S$,
and by $\mathcal{C}_b^\delta(S)$ the space of bounded and continuous real-valued 
functions on $S$ with finite norm $|f|_{\delta}$, for any $\delta\ge 0$. Furthermore, we
set $\mathcal{C}_b^0(S)\equiv\mathcal{C}_b(S)$ and also denote by
$\mathcal{C}_b^{\infty}(S)$ the space of smooth real-valued functions on $S$
with bounded derivatives of any order.

Throughout this paper we assume the following conditions for equations
(\ref{PDE_1})-(\ref{terminal}).

%%%%%%%%%%%%%%%%%%%%%%%%%%%%%%%%%%%%%%%%%%%%%%%%%%%%%%%%%%%%%%%%%%%5

\begin{assumption}\label{data assumption}
(i) The $n\times d$ matrix-valued diffusion coefficient $\sigma$,
the $\mathbb{R}^{n}$-valued drift coefficient $b$, the real-valued obstacle $f$ and terminal
datum $U$ have finite norms
$|\sigma|_{1}, |b|_{1}, |f|_{1}, |U|_{1}\le M$ for some $M>0$. Moreover, 
%\textcolor{red}{$[f]_{1,1}\le M$} and 
$f(T,\cdot)\ge U$ in $\mathbb{R}^{n}$.

(ii) The Hamiltonian $H(t,x,p)\in \mathcal{C}(Q_{T}\times\mathbb{R}^n)$ is %first order differentiable,
convex in p, and satisfies the coercive condition
$$\lim_{|p|\rightarrow\infty}\frac{H(t,x,p)}{|p|}=+\infty,$$
uniformly in $(t,x)\in Q_{T}$. Moreover,
for every $p$, $[H(\cdot,\cdot,p)]_{1}\le M$ for the constant $M$ in (i), and there exist two
locally bounded functions $H^{*}$ and
$H_{*}:\mathbb{R}^{n}\to\mathbb{R}$ such that
$$H_{*}(p)=\inf_{(t,x)\in Q_{T}}H(t,x,p), \ \ \  H^{*}(p)=\sup_{(t,x)\in Q_{T}}H(t,x,p).$$
\end{assumption}

Unless state otherwise, we will then throughout this paper denote by $C:=C(T,M)$ some constant that depends only on $T$ and $M$. Then under the above assumptions, we have the following existence,
uniqueness and regularity results for equation (\ref{PDE_1})-(\ref{terminal}). Their
proofs are provided in Appendix A.

\begin{proposition}\label{solutionproperty} Suppose that Assumption \ref{data assumption} is
satisfied. Then, there exists a unique viscosity solution
$u\in\mathcal{C}_b^1(\bar{Q}_T)$ of
(\ref{PDE_1})-(\ref{terminal}), with $|u|_{1}\le C$.
%for some
%constant $C$  depending only on $M$ and $T$.
\end{proposition}

%Note that the semilinear PDE (\ref{PDE_1})-(\ref{terminal}) is a
%special case (choosing $\varepsilon=0$) of the HJB equation
%(\ref{PtbedPDE})-(\ref{Ptbed_terminal}), so we omit the proof of
%Proposition \ref{solutionproperty}, and prove a more general result
%in Proposition \ref{upperbound_1}.

%%%%%%%%%%%%%%%%%%%%%%%%%%%%%%%%%%%%%%%%%%%%%%%%%%%%%%%%%%%%%%%%%%%%5
\subsection{The backward operator $\mathbf{S}_t(\Delta)$}

Before introducing the approximation scheme, we introduce a \emph{backward operator}
$\mathbf{S}_{t}(\Delta)$, which is defined in \cite{Huang}. Herein, for the reader's convenience, we repeat the definition. To this end, using the convexity and coerciveness of the Hamiltonian $H(t,x,p)$, we define its Legendre transform  $L: Q_{T}\times\mathbb{R}^{n}\to\mathbb{R}$ by
\begin{equation}\label{L}
L(t,x,q):=\sup_{p\in\mathbb{R}^n}\{p\cdot q-H(t,x,p)\}.
\end{equation}

Next, for any $t$ and $\Delta$ such that $0\leq t<t+\Delta \leq T$ and any $\phi\in\mathcal{C}_b(\mathbb{R}^{n})$, the \emph{backward operator}
$\mathbf{S}_{t}(\Delta):\mathcal{C}_b(\mathbb{R}^{n})\to\mathcal{C}_b(\mathbb{R}^{n})$ is defined by
\begin{eqnarray}\label{semigroupequation1}
\left\{
\begin{array}{ll}
\displaystyle \mathbf{S}_{t}(\Delta)\phi(x)=\min_{y\in\mathbb{R}^n} \left\{\Delta L\left(t,x,\frac{x-y}{\Delta}\right)+\mathbf{E}[\phi(Y_{t+\Delta}^{t,y})|\mathcal{F}_t]\right\},& x\in\mathbb{R}^{n}\\
\displaystyle Y_{s}^{t,y}=y+ b(t,y)(s-t)+\sigma(t,y)(W_{s}-W_t),& s\in[t,t+\Delta],
\end{array}
\right.
\end{eqnarray}
on a filtered probability space
$(\Omega,\mathcal{F},\{\mathcal{F}_t\}_{t\geq 0},\mathbf{P})$, where
$W$ is an $d$-dimensional Brownian motion with its augmented
filtration $\{\mathcal{F}_t\}_{t\geq 0}$.\\

Note that in the definition of the operator $\mathbf{S}_t(\Delta)$, it is implied that for any $\phi\in\mathcal{C}_b(\mathbb{R}^{n})$, there always exists an associated minimizer $y^*$ and $\mathbf{S}_{t}(\Delta)\phi$ is also in $\mathcal{C}_b(\mathbb{R}^{n})$. These along with other key properties of the backward operator $\mathbf{S}_{t}(\Delta)$ are proved in \cite{Huang}.

\subsection{The approximation scheme}\label{sebsection: splitting}

We are now in a position to introduce the approximation scheme for the variational inequality (\ref{PDE_1})-(\ref{terminal}). This approximation scheme is a natural extension of the scheme (2.11) in \cite{Huang}: For $\Delta\in(0,T)$ and $(t,x)\in\bar{Q}_{T-\Delta}$, we introduce the iterative
algorithm
\begin{equation}\label{splitting}
u^{\Delta}(t,x)=\min\left\{\mathbf{S}_{t}(\Delta)u^{\Delta}(t+\Delta,\cdot)(x),f(t,x)\right\}
\end{equation}
with $u^{\Delta}(T,\cdot)=U(\cdot)$ and $\mathbf{S}_t(\Delta)$
defined in (\ref{semigroupequation1}). The values between $T-\Delta$
and $T$ are obtained by a standard linear interpolation.

Specifically, the approximation scheme is given by
\begin{eqnarray}\label{semischeme}
\left\{
\begin{array}{ll}
\displaystyle \bar{S}(\Delta, t,x,u^\Delta(t,x),u^\Delta(t+\Delta,\cdot))=0 &\text{in}\  \bar{Q}_{T-\Delta},\\
\displaystyle u^\Delta(t,x)=g^\Delta(t,x)&\text{in}\
\bar{Q}_{T}\backslash\bar{Q}_{T-\Delta},
\end{array}
\right.
\end{eqnarray}
where
$\bar{S}:(0,T]\times\bar{Q}_{T-\Delta}\times\mathbb{R}\times\mathcal{C}_b(\mathbb{R}^n)\rightarrow\mathbb{R}$,
and $g^\Delta:\bar{Q}_{T}\backslash \bar{Q}_{T-\Delta}\to\mathbb{R}$
are defined respectively by
\begin{equation}\label{Sbar1}
\bar{S}(\Delta,t,x,p,v)=\max\{S(\Delta,t,x,p,v),p-f(t,x)\},
\end{equation}
\begin{equation}\label{Sbar2}
S(\Delta,t,x,p,v)=\frac{p-\mathbf{S}_{t}(\Delta)v(x)}{\Delta},
\end{equation}and
\begin{equation}\label{gdelta}
g^\Delta(t,x)=\omega_1(t)U(x)+\omega_2(t)\min\{\mathbf{S}_{T-\Delta}(\Delta)U(x),f(T-\Delta,x)\},
\end{equation}
with $\omega_1(t)=(t+\Delta-T)/\Delta$ and
$\omega_2(t)=(T-t)/\Delta$ being the linear interpolation weights.

Note that when $T-\Delta<t\le T$, the approximate term
$g^\Delta$ corresponds to the usual linear interpolation between
$T-\Delta$ and $T$. 

The next proposition shows the well-posedness of the
approximation scheme (\ref{semischeme}). The proof is almost the same as Lemma 2.6 in \cite{Huang} so we omit the detail here. We remark that, unlike the
viscosity solution $u$ of the variational inequality
(\ref{PDE_1})-(\ref{terminal}), the solution $u^{\Delta}$ of the
approximation scheme does not in general have enough regularity.

\begin{proposition} %(Convergence of approximation scheme)
\label{semigrouptheorem} Suppose that Assumption 2.1 is satisfied and let $\Delta\in(0,T)$.
Then, the approximation scheme (\ref{semischeme}) admits a unique
solution $u^\Delta\in\mathcal{C}_b(\bar{Q}_{T})$ with
$|u^{\Delta}|_{0}\leq C$.
% where the constant $C$  depends only on $M$
%and $T$.
\end{proposition}
%\begin{proof}
%By the stability property (iv) of Proposition \ref{semigroup}, we
%know that $\mathbf{S}_{t}(\Delta)\phi$ is uniformly bounded if so is
%$\phi$. This, together with the fact that $f$ is bounded by $M$, guarantees that equation (\ref{splitting}) is always well defined
%in $\bar{Q}_{T-\Delta}$, which implies the existence and uniqueness
%of the solution $u^\Delta$. Furthermore, for $0\leq t\leq T-\Delta$,
%$$|u^{\Delta}(t,\cdot)|_{0}\leq \max\{C\Delta+|u^{\Delta}(t+\Delta,\cdot)|_{0},M\}\leq C\Delta+\max\{|u^{\Delta}(t+\Delta,\cdot)|_{0},M\},$$
%and thus
%$$\max\{|u^{\Delta}(t,\cdot)|_{0},M\}\leq C\Delta+\max\{|u^{\Delta}(t+\Delta,\cdot)|_{0},M\},$$
%where $C=\max\left\{|L^{*}(0)|,|H^{*}(0)|\right\}$.
%By backward induction and the definition of $g^{\Delta}$ in (\ref{semischeme}), we conclude that
%$$|u^{\Delta}|_{0}\leq CT+\max\{\sup_{t\in(T-\Delta,T]}|g^{\Delta}(t,\cdot)|_{0},M\}\leq C,$$
%where the constant $C$ only depends on $M$ and $T$.
%\end{proof}

Thanks to the properties of $S(\Delta,t,x,p,v)$ established in Proposition 2.8 of \cite{Huang}, we immediately obtain the following key
properties of the approximation scheme (\ref{semischeme}).

\begin{proposition}\label{scheme property}
Suppose that Assumption \ref{data assumption} is satisfied and let $\Delta\in(0,T)$, $(t,x)\in \bar{Q}_{T-\Delta}$, $p\in\mathbb{R}$ and $v\in\mathcal{C}_{b}(\mathbb{R}^{n})$. Then, the approximation scheme
$\bar{S}(\Delta,t,x,p,v)$ has the following properties:

(i) (Monotonicity) For any $c_{1}, c_{2}\in\mathbb{R}$, and any function
$u\in\mathcal{C}_{b}(\mathbb{R}^{n})$ with $u\le v$,
$$\bar{S}(\Delta,t,x,p+c_{1},u+c_{2})\ge \bar{S}(\Delta,t,x,p,v)+\min\{\frac{c_{1}-c_{2}}{\Delta},c_{1}\}.$$

(ii) (Consistency) For any
$\phi\in\mathcal{C}_b^{\infty}(\bar{Q}_{T})$,
\begin{align}\label{consistant_error}
 &\ |\max\{-\partial_t\phi+g(t,x,\partial_{x}\phi,\partial_{xx}\phi),\phi-f(t,x)\}-\bar{S}(\Delta,t,x,\phi,\phi(t+\Delta,\cdot))|\notag\\
 \le  &\
 C\Delta\left(|\partial_{tt}\phi|_{0}+|\partial_{xxxx}\phi|_0+|\partial_{xxt}\phi|_{0}+\mathcal{R}(\phi)\right)\ \ \ \text{in}\  \bar{Q}_{T-\Delta},
\end{align}
where the constant $C$ depends only on $[\phi]_{2,1}$, $M$ and $T$,
and $\mathcal{R}(\phi)$ represents the ``insignificant'' terms
containing the lower order derivatives of $\phi$.
\end{proposition}

\begin{proof}
(i) follows immediately from the definition of $\bar{S}$, (\ref{Sbar1})-(\ref{Sbar2}), and that of $\mathbf{S}_{t}(\Delta)$, (\ref{semigroupequation1}). (ii) follows from Proposition 2.8 (iv) in \cite{Huang}.
\end{proof}\\

The monotonicity property (i) in Proposition \ref{scheme property} then implies the following comparison result for the approximation scheme (\ref{semischeme}), which will be used throughout this paper. The proof is analogous to Proposition 2.9 of \cite{Huang}, with a slight difference to accommodate the extension to the variational inequality case.

\begin{proposition}\label{schemecomparison}
Suppose that Assumption 2.1 is satisfied, and that $u$,
$v\in\mathcal{C}_b(\bar{Q}_T)$ are such that
$$\bar{S}(\Delta,t,x,u,u(t+\Delta,\cdot)) \leq h_1\ \text{in}\  \bar{Q}_{T-\Delta},$$
$$\bar{S}(\Delta,t,x,v,v(t+\Delta,\cdot)) \ge h_2\ \text{in}\  \bar{Q}_{T-\Delta},$$
for some $h_1$, $h_2\in\mathcal{C}_b(\bar{Q}_{T-\Delta})$. Then,
\begin{equation}
u-v\leq \sup_{\bar{Q}_{T}\backslash
\bar{Q}_{T-\Delta}}(u-v)^{+}+(T-t+1)\sup_{\bar{Q}_{T-\Delta}}(h_{1}-h_{2})^{+}\
\text{in}\ \bar{Q}_{T}.
\end{equation}
\end{proposition}

\begin{proof}
Without loss of generality, we assume that
\begin{equation}\label{comparisonforscheme}
u\le v\ \text{in}\ \bar{Q}_{T}\backslash \bar{Q}_{T-\Delta}\ \
\text{and}\ \ h_{1}\leq h_{2}\ \text{in}\ \bar{Q}_{T-\Delta},
\end{equation}
since, otherwise, the function $w:=v+\sup_{\bar{Q}_{T}\backslash
\bar{Q}_{T-\Delta}}(u-v)^{+}+(T-t+1)\sup_{\bar{Q}_{T-\Delta}}(h_{1}-h_{2})^{+}$
satisfies that $u\le w$ in $\bar{Q}_{T}\backslash
\bar{Q}_{T-\Delta}$ and by the monotonicity property (i) in Proposition
\ref{scheme property},
\begin{align*}
\bar{S}(\Delta,t,x,w,w(t+\Delta,\cdot)) &\ge
 \bar{S}(\Delta,t,x,v,v(t+\Delta,\cdot))+\sup_{\bar{Q}_{T-\Delta}}(h_{1}-h_{2})^{+}\\
&\ge h_{2}+\sup_{\bar{Q}_{T-\Delta}}(h_{1}-h_{2})^{+} \ge h_{1}\
\text{in}\ \bar{Q}_{T-\Delta}.
\end{align*}
%which, together with the fact that $w\ge v+\sup_{\bar{Q}_{T-\Delta}}(h_{1}-h_{2})^{+}$, implies that 
%$$\bar{S}(\Delta,t,x,w,w(t+\Delta,\cdot))\ge h_{1}\ \text{in}\ \bar{Q}_{T-\Delta}.$$
Thus, it suffices to prove $u\le v$ in $\bar{Q}_{T}$ when
(\ref{comparisonforscheme}) holds.

To this end, for $b\ge 0$, let $\psi_{b}(t):=b(T-t)$ and
$M(b):=\sup_{\bar{Q}_{T}}\{u-v-\psi_{b}\}.$ Then our goal becomes to
prove $M(0)\leq 0$ and we prove by contradiction. Assume $M(0)>0$, then by the continuity of $M$,
we must have $M(b)>0$ for some $b>0$. For such $b$, take a sequence
$\{(t_{n},x_{n})\}_{n}$ in $\bar{Q}_{T}$ such that $\delta_{n}:=M(b)-(u-v-\psi_{b})(t_{n},x_{n})\downarrow 0,\ \ \text{as}\ n\to\infty.$
Since $M(b)>0$ but $u-v-\psi_{b}\leq 0$ in $\bar{Q}_{T}\backslash \bar{Q}_{T-\Delta}$, we must have $t_{n}\leq T-\Delta$ for sufficiently large $n$. Then for such $n$, we use the monotonicity property (i) in Proposition \ref{scheme property} again to obtain  
\begin{align*}
h_1(t_{n},x_{n}) \ge &\ \bar{S}(\Delta,t_{n},x_{n},u(t_{n},x_{n}),u(t_{n}+\Delta,\cdot)) \\
 \ge &\
 \bar{S}(\Delta,t_{n},x_{n},v(t_{n},x_{n})+\psi_{b}(t_{n})+M(b)-\delta_{n},v(t_{n}+\Delta,\cdot)+\psi_{b}(t_{n}+\Delta)+M(b)) \\
 \ge &\
 \bar{S}(\Delta,t_{n},x_{n},v(t_{n},x_{n}),v(t_{n}+\Delta,\cdot))+\min\{b-\delta_{n}\Delta^{-1},\psi_{b}(t_{n})+M(b)-\delta_{n}\} \\
 \ge &\
 h_{2}(t_{n},x_{n})+\min\{1,\Delta\}(b-\delta_{n}\Delta^{-1}),
\end{align*}
where the last inequality follows from $M(b)>0$ and $\psi_{b}(t_{n})\ge \Delta b$.
Since $h_{1}\leq h_{2}$ in $\bar{Q}_{T-\Delta}$, we then must have
$b-\delta_{n}\Delta^{-1}\leq 0$. Thus, we deduce $b\leq 0$ by
letting $n\to\infty$, which is a contradiction.
\end{proof}

%%%%%%%%%%%%%%%%%%%%%%%%%%%%%%%%%%%%%%%%%
%%%%%%%%%%%%%%%%%%%%%%%%%%%%%%%%%%%%%%%%%%%%%

%%%%%%%%%%%%%%%%%%%%%%%%%%%%%%%%%%%%%%%%%%%%%%%%%%%%%%%%%%%%%%%%%%%%%%
%%%%%%%%%%%%%%%%%%%%%%%%%%%%%%%%%%%%%%%%%%%%%%%%%%%%%%%%%%%%%%%%%
\section{Convergence rate of the  approximation scheme}

In this section, we establish the (uniform) convergence rate of the approximate solution
$u^{\Delta}$ to the viscosity solution $u$ of the variational inequality
(\ref{PDE_1})-(\ref{terminal}), which is the main result of this paper. To this end, we shall derive a (uniform)  bound for the approximation error $u-u^{\Delta}$ in $\bar{Q}_{T}$.

We start with the approximation error in the last time interval $\bar{Q}_{T}\backslash \bar{Q}_{T-\Delta}$, where the value of $u^{\Delta}$ involves only a one-time approximation and some linear interpolation. Therefore, the bound for the approximation error in this domain can be easily obtained by some properties of the backward operator $\mathbf{S}_t(\Delta)$ and some regularity results of $u$. This is demonstrated in the following lemma.

\begin{lemma}\label{errorsmall}
Suppose that Assumption \ref{data assumption} is satisfied. Let $\Delta\in(0,T)$,
$u^{\Delta}\in\mathcal{C}_b(\bar{Q}_{T})$ be the unique solution of the approximation scheme
(\ref{semischeme}) and $u\in\mathcal{C}_b^{1}(\bar{Q}_{T})$ be the unique viscosity solution of equation (\ref{PDE_1})-(\ref{terminal}). Then, 
%there exists a
%constant $C$ only depending on $M$ and $T$ such that
\begin{equation}\label{estimate_final_interval}
\sup_{\bar{Q}_{T}\backslash \bar{Q}_{T-\Delta}}|u-u^{\Delta}|\leq C\sqrt{\Delta}.
\end{equation}
\end{lemma}

\begin{proof}
From the property
(\emph{v}) of the operator $\mathbf{S}_{t}(\Delta)$ (cf. Proposition 2.5 in \cite{Huang}), we have
$|U-\mathbf{S}_{T-\Delta}(\Delta)U|_{0}\le C\sqrt{\Delta}.$
%for some constant $C$ only depending on $M$ and $T$.
On the other hand, by Assumption \ref{data assumption}(i), for any $x\in\mathbb{R}^{n}$,
$f(T-\Delta,x)\ge f(T,x)-[f]_{1,1/2}\sqrt{\Delta}\ge U(x)-M\sqrt{\Delta}.$
The above two inequalities then imply that for any $x\in\mathbb{R}^{n}$,
\begin{equation}\label{est}
|U(x)-\min\{\mathbf{S}_{T-\Delta}(\Delta)U(x),f(T-\Delta,x)\}|\le C\sqrt{\Delta}.
\end{equation}
Then from (\ref{semischeme}), we have, for $(t,x)\in \bar{Q}_{T}\backslash \bar{Q}_{T-\Delta}$, 
\begin{align*}
|u(t,x)-u^{\Delta}(t,x)|= &\ |u(t,x)-g^{\Delta}(t,x)|\\
=&\ |u(t,x)-u(T,x)+\omega_{2}(t)(U(x)-\min\{\mathbf{S}_{T-\Delta}(\Delta)U(x),f(T-\Delta,x)\})| \\
 \leq &\
 |u(t,x)-u(T,x)|+|U(x)-\min\{\mathbf{S}_{T-\Delta}(\Delta)U(x),f(T-\Delta,x)\}| \\
 \leq &\
 C(\sqrt{|T-t|}+\sqrt{\Delta}) \leq C\sqrt{\Delta},
\end{align*}
where the
second to last inequality follows from the regularity property of the
solution $u$ (cf. Proposition \ref{solutionproperty}) and (\ref{est}).
\end{proof}\\

Next, we derive a bound of approximation error within the whole domain $\bar{Q}_{T}$. We first consider a special case when (\ref{PDE_1})-(\ref{terminal}) admits a unique smooth solution $u$ with
bounded derivatives of any order. 

\begin{theorem} Suppose that Assumption \ref{data assumption} is
satisfied. Let $\Delta\in(0,T)$ and 
$u^{\Delta}\in\mathcal{C}_b(\bar{Q}_{T})$ be the unique solution of the approximation scheme
(\ref{semischeme}). Suppose that equation
(\ref{PDE_1})-(\ref{terminal}) admits a unique smooth solution
$u\in\mathcal{C}_b^{\infty}(\bar{Q}_T)$. Then,
% there exists a constant $C$ only depending on $M$ and $T$ such that
$$|u-u^{\Delta}|\leq C\sqrt{\Delta}\ \ \text{in}\ \bar{Q}_T.$$
\end{theorem}

\begin{proof} Using that $u\in\mathcal{C}_b^{\infty}(\bar{Q}_T)$, the consistency error
estimate (\ref{consistant_error}) yields
\begin{align*}
&|\bar{S}(\Delta,t,x,u(t,x),u(t+\Delta,\cdot))|\\
\leq&\
C\Delta\left(|\partial_{tt}u|_{0}+|\partial_{xxxx}u|_0+|\partial_{xxt}u|_{0}+\mathcal{R}(u)\right)
\leq C\Delta,
\end{align*}
 for
$(t,x)\in\bar{Q}_{T-\Delta}$. 
%From the fact that $u$ is the solution of (\ref{PDE_1})-(\ref{terminal}),
%\begin{align*}
%&\ \bar{S}(\Delta,t,x,u(t,x),u(t+\Delta,\cdot)) \\
%\le &\  \max\{-\partial_tu(t,x)+g(t,x,\partial_{x}u(t,x),\partial_{xx}u(t,x)+C\Delta, u(t,x)-f(t,x)\}\le C\Delta,
%\end{align*}
%and similarly 
%$$\bar{S}(\Delta,t,x,u(t,x),u(t+\Delta,\cdot))\ge -C\Delta.$$
 On the other hand, from the definition
of the approximation scheme (\ref{semischeme}), we have
$$\bar{S}(\Delta,t,x,u^{\Delta}(t,x),u^{\Delta}(t+\Delta,\cdot))=0,$$
for $(t,x)\in\bar{Q}_{T-\Delta}$. In turn, the comparison principle
in Proposition \ref{schemecomparison} implies
$$|u-u^{\Delta}|\leq \sup_{\bar{Q}_T\backslash\bar{Q}_{T-\Delta}}|u-u^{\Delta}|+(T-t+1)C\Delta\ \ \text{in}\ \bar{Q}_T.$$
%and
%$$u^{\Delta}(t,x)-u(t,x)\leq \sup_{(t,x)\in\bar{Q}_T\backslash\bar{Q}_{T-\Delta}}(u^{\Delta}(t,x)-u(t,x))^++(T-t+1)C\Delta,$$
%for $(t,x)\in\bar{Q}_T$. 
The conclusion then follows by using the
estimate (\ref{estimate_final_interval}) in Lemma \ref{errorsmall}.
\end{proof}\\

In general, since (\ref{PDE_1})-(\ref{terminal}) only
admits a viscosity solution $u\in\mathcal{C}_b^1(\bar{Q}_T)$ (cf.  Proposition \ref{solutionproperty}) due to
the possible degeneracies of the equation, the above result does not
hold. A natural idea is then to approximate the viscosity solution
$u$ by a sequence of smooth sub- and supersolutions
$u_{\varepsilon}$ and, in turn, compare them with $u^{\Delta}$
using the comparison result for the approximation scheme (cf. Proposition \ref{schemecomparison}) to obtain the upper and lower bound for the approximation error seperately. We carry
out this {procedure} next.

%%%%%%%%%%%%%%%%%%%%%%%%%%%%%%%%%%%%%%%%%%%%%%%%%%%

%%%%%%%%%%%%%%%%%%%%%%%%%%%%%
\subsection{Upper bound for the approximation {error}}\label{upbd}

We now derive an upper bound for the approximation error within the whole domain $\bar{Q}_{T}$ for the general $u\in\mathcal{C}_b^1(\bar{Q}_T)$ case. We first construct a sequence of smooth subsolutions to
(\ref{PDE_1}) by perturbing its coefficients. This approach, known as the \emph{shaking
coefficients technique}, was initially proposed by Krylov \cite{Krylov1}
\cite{Krylov}, and further developed by Barles and Jakobsen
\cite{BJ1} \cite{Jakobsen}. We apply this approach to obtain an upper bound for the approximation error
$u-u^{\Delta}$.

To this end, for small enough $\varepsilon\ge 0$, we extend the functions $f$ and $\eta:=\sigma, b$  to
${Q}^{-\varepsilon^{2}}_{T+\varepsilon^{2}}:=[-\varepsilon^{2},T+\varepsilon^{2})\times\mathbb{R}^{n}$
and $H$ to
${Q}^{-\varepsilon^{2}}_{T+\varepsilon^{2}}\times\mathbb{R}^{n}$, so that Assumption \ref{data assumption} still holds.
We then define $\eta^\theta(t,x):=\eta(t+\tau,x+e)$ and
$H^{\theta}(t,x,p):=H(t+\tau,x+e,p)$, where $\theta=(\tau,e)$ with
$\theta\in\Theta^\varepsilon:=[-\varepsilon^2,0]\times\varepsilon
B(0,1)$. We then consider the perturbed version of
(\ref{PDE_1})-(\ref{terminal}), namely,
\begin{align}
\max\{-\partial_tu^{\varepsilon}+\sup_{\theta\in\Theta^{\varepsilon}}g^{\theta}(t,x,\partial_{x}u^{\varepsilon},\partial_{xx}u^{\varepsilon}),u^{\varepsilon}-f(t,x)\}&=0 & \text{in}\ Q_{T+{\varepsilon}^{2}};\label{PtbedPDE}\\
u^{\varepsilon}(T+{\varepsilon}^{2},x)&=U(x) & \text{in}\ \mathbb{R}^n,\label{Ptbed_terminal}
\end{align}
where
$$g^{\theta}(t,x,p,X)=-\frac{1}{2}\text{Trace}\left(\sigma^{\theta}{\sigma^{\theta}}^T(t,x)X\right)-b^{\theta}(t,x)\cdot p+H^{\theta}(t,x,p).$$
%where $\theta=(\tau,e)$,
%$\Theta^\varepsilon=[-\varepsilon^2,0]\times\varepsilon B(0,1)$,
%$\eta^\theta(\cdot,\cdot)=\eta(\cdot+\tau,\cdot+e)$ for $\eta=\sigma,b$, and
%$H^{\theta}(\cdot,\cdot,\cdot)=H(\cdot+\tau,\cdot+e,\cdot)$. The
%functions $\eta$, $H$ and $f$ are appropriately extended, respectively, to
%${Q}^{-\varepsilon^{2}}_{T+\varepsilon^{2}}:=[-\varepsilon^{2},T+\varepsilon^{2})\times\mathbb{R}^{n}$
%and ${Q}^{-\varepsilon^{2}}_{T+\varepsilon^{2}}\times\mathbb{R}^{n}$ such
%that Assumption \ref{data assumption} still holds.
Note that by letting the perturbation parameter $\varepsilon=0$, we can retrieve our original variation inequality (\ref{PDE_1})-(\ref{terminal}). 

We establish existence, uniqueness and regularity results for the perturbed equation
(\ref{PtbedPDE})-(\ref{Ptbed_terminal}), and a comparison between
$u$ and $u^{\varepsilon}$ in the next proposition, whose proof is provided in Appendix \ref{App1}.

\begin{proposition}\label{upperbound_1}
Suppose that Assumption \ref{data assumption} is satisfied. Then, for small enough $\varepsilon\ge 0$,
there exists a unique viscosity solution
$u^{\varepsilon}\in\mathcal{C}^{1}_b(\bar{Q}_{T+\varepsilon^2})$ of (\ref{PtbedPDE})-(\ref{Ptbed_terminal}), with
$|u^{\varepsilon}|_{1}\leq C$. 
%for some constant $C$ only depending on $M$ and $T$. 
Moreover,
\begin{equation}\label{est2}
 {|u-u^{\varepsilon}|}\leq C\varepsilon\ \ \text{in}\ \bar{Q}_T.
\end{equation}
\end{proposition}

Next, we  regularize $u^{\varepsilon}$ by
a standard mollification procedure. For this, let $\rho(t,x)$ be an $\mathbb{R}_+$-valued
smooth function with support in $(-1,0)\times B(0,1)$
and mass $1$, and introduce the sequence of mollifiers
$\rho_{\varepsilon}$ for $\varepsilon>0$,
\begin{equation}\label{mollifer}
\rho_{\varepsilon}(t,x):=\frac{1}{\varepsilon^{n+2}}\rho\left(\frac{t}{\varepsilon^2},\frac{x}{\varepsilon}\right).
\end{equation}
For $(t,x)\in \bar{Q}_T$, we then define
$$u_{\varepsilon}(t,x)=u^{\varepsilon}*
\rho_{\varepsilon}(t,x)=\int_{-\varepsilon^2< \tau< 0}\int_{|e|<
\varepsilon}u^{\varepsilon}(t-\tau,x-e)\rho_{\varepsilon}(\tau,e)ded\tau.$$
Standard properties of mollifiers imply that
$u_{\varepsilon}\in\mathcal{C}_b^{\infty}(\bar{Q}_{T})$,
\begin{equation}\label{upperbound_2}
|u^{\varepsilon}-u_{\varepsilon}|_0\leq C\varepsilon,
\end{equation}
and, moreover, for positive integer $i$ and multiindex $j$,
\begin{equation}\label{mollifier}
|\partial_{t}^iD_{x}^ju_{\varepsilon}|_0\leq C\varepsilon^{1-2i-|j|}.
\end{equation}

We observe from (\ref{PtbedPDE}) that the function $v_{\theta}^{\varepsilon}(t,x):=u^{\varepsilon}(t-\tau,x-e)$ is a viscosity subsolution of 
\begin{equation}\label{pde1}
-\partial_{t}w(t,x)+g(t,x,\partial_{x}w(t,x),\partial_{xx}w(t,x))=0,
\end{equation}
in ${Q}_{T}$ for any
{$\theta\in\Theta^{\varepsilon}$}. On the other hand, a Riemann sum approximation shows that $u_{\varepsilon}(t,x)$ can be
viewed as the limit of convex combinations of
$v_{\theta}^{\varepsilon}(t,x)$ for {$\theta\in\Theta^{\varepsilon}$}. Since the nonlinear term
$g(t,x,p,X)$ is convex in $p$ and linear in
$X$, the convex combinations of
$v_{\theta}^{\varepsilon}(t,x)$ are also subsolutions of (\ref{pde1}) in ${Q}_{T}$. Using the stability of viscosity solutions, we deduce that $u_{\varepsilon}(t,x)$ is still a subsolution
of (\ref{pde1}) in $Q_T$, namely,
\begin{equation}\label{sub2}
-\partial_{t}u_{\varepsilon}+g(t,x,\partial_{x}u_{\varepsilon},\partial_{xx}u_{\varepsilon})\le 0.
\end{equation}
%\begin{proof}
%We split the consistent error
%$\Delta\mathcal{E}(t,\Delta,u_{\varepsilon})$ into two parts as follows:
%\begin{align*}
%&|u_{\varepsilon}(t-\Delta,x)-\mathbf{S}_{t-\Delta}(\Delta)u_{\varepsilon}(t,x)+\Delta\partial_tu_{\varepsilon}(t,x)-\Delta\mathbf{L}_t u_{\varepsilon}(t,x)|_0\\
%\leq &\
%|u_{\varepsilon}(t,x)-\mathbf{S}_{t-\Delta}(\Delta)u_{\varepsilon}(t,x)-\Delta\mathbf{L}_t u_{\varepsilon}(t,x)|_0
%+
%|u_{\varepsilon}(t,x)-u_{\varepsilon}(t-\Delta,x)-\Delta\partial_tu_{\varepsilon}(t,x)|_0.
%\end{align*}
%
%The estimate of the first term has been shown in (vi) of Lemma
%\ref{semigroup}. Furthermore, by applying the estimates of the
%mollifiers (\ref{mollifier}) and keeping the worst terms involving
%$\varepsilon$, we obtain that
%\begin{equation}\label{estimate5}
%|u_{\varepsilon}(t,x)-\mathbf{S}_{t-\Delta}(\Delta)u_{\varepsilon}(t,x)-\Delta\mathbf{L}_t u_{\varepsilon}(t,x)|_0\leq C\Delta^2\varepsilon^{-3}.
%\end{equation}
%
%For the second term, using Taylor expansion, we get
%\begin{align}\label{estimate9}
%&\
%|u_{\varepsilon}(t,x)-u_{\varepsilon}(t-\Delta,x)-\Delta\partial_tu_{\varepsilon}(t,x)|_0\notag\\
%\leq&\
%|\int_{t-\Delta}^{t}\left(\partial_{t}u_{\varepsilon}(t,x)-\int_v^{t}\partial_{tt}u_{\varepsilon}(u,x)du\right)dv-
%\Delta\partial_tu_{\varepsilon}(t,x)|_0\notag\\
%\leq&\ C\Delta^2|\partial_{tt}u_{\varepsilon}|_0\leq
%C\Delta^2\varepsilon^{-3}.
%\end{align}
%
%Hence, the consistent error estimate (\ref{estimate4}) follows from
%(\ref{estimate5}) and (\ref{estimate9}).
%\end{proof}

We are now in a position to establish an upper bound for the approximation error.

\begin{theorem}\label{theorem_error_1}
Suppose that Assumption \ref{data assumption} is satisfied. Let $\Delta\in(0,T)$,
$u^{\Delta}\in\mathcal{C}_b(\bar{Q}_{T})$ be the unique solution of the approximation scheme
(\ref{semischeme}) and $u\in\mathcal{C}_b^{1}(\bar{Q}_{T})$ be the unique viscosity solution of equation (\ref{PDE_1})-(\ref{terminal}). Then, 
$$u-u^{\Delta}\leq C\Delta^{\frac14}\ \text{in}\ \bar{Q}_{T}.$$
\end{theorem}

\begin{proof}
From (\ref{PtbedPDE}), $u^{\varepsilon}-f\le 0$ in $Q_{T+\varepsilon^{2}}$. This yields that for $(t,x)\in\bar{Q}_{T-\Delta}$,
$$u_{\varepsilon}(t,x)-f(t,x)\le \int_{-\varepsilon^2< \tau< 0}\int_{|e|<
\varepsilon}(f(t-\tau,x-e)-f(t,x))\rho_{\varepsilon}(\tau,e)ded\tau\le [f]_{1}\varepsilon \le M\varepsilon.$$
This together with (\ref{sub2}) gives that in $\bar{Q}_{T-\Delta}$,
$$\max\{-\partial_tu_{\varepsilon}+g(t,x,\partial_{x}u_{\varepsilon},\partial_{xx}u_{\varepsilon}),u_{\varepsilon}-f(t,x)\}\le M\varepsilon.$$
We then substitute $u_{\varepsilon}$ into the consistency error estimate
(\ref{consistant_error}) and use the estimate (\ref{mollifier}) to obtain
%\begin{align*}\label{estimate4}
%%\mathcal{E}(t,\Delta,u_{\varepsilon}):=&\
%&\left|-\partial_tu_{\varepsilon}(t,x)+g(t,x,\partial_{x}u_{\varepsilon},\partial_{xx}u_{\varepsilon})(t,x)-S(\Delta,t,x,u_{\varepsilon}(t,x),u_{\varepsilon}(t+\Delta,\cdot))\right|\notag\\
%\leq&\
%C\Delta\left(|\partial_{tt}u_{\varepsilon}|_{0}+|\partial_{xxxx}u_{\varepsilon}|_0+|\partial_{xxt}u_{\varepsilon}|_{0}+\mathcal{R}(u_{\varepsilon})\right)
%\leq C\Delta\varepsilon^{-3},
%\end{align*}
%for $(t,x)\in\bar{Q}_{T-\Delta}$. Since $u_{\varepsilon}$ is a
%subsolution of (\ref{pde1}), we further have
%\begin{equation*}
%S(\Delta,t,x,u_{\varepsilon}(t,x),u_{\varepsilon}(t+\Delta,\cdot))\leq
%C\Delta\varepsilon^{-3},
%\end{equation*}
%for $(t,x)\in\bar{Q}_{T-\Delta}$. 
\begin{align*}\label{estimate4}
%\mathcal{E}(t,\Delta,u_{\varepsilon}):=&\
&\bar{S}(\Delta,t,x,u_{\varepsilon}(t,x),u_{\varepsilon}(t+\Delta,\cdot))\\
\leq M\varepsilon+C\Delta\left(|\partial_{tt}u_{\varepsilon}|_{0}+\right.& |\partial_{xxxx}u_{\varepsilon}|_0+|\partial_{xxt}u_{\varepsilon}|_{0}+\mathcal{R}(u_{\varepsilon}))
\leq C(\varepsilon+\Delta\varepsilon^{-3}),
\end{align*}
for $(t,x)\in\bar{Q}_{T-\Delta}$.
%Furthermore, by the
%definition of the splitting scheme (\ref{semischeme}), we also have
%$$\bar{S}(\Delta,t,x,u^{\Delta}(t,x),u^{\Delta}(t+\Delta,\cdot))=0,$$
%for $(t,x)\in\bar{Q}_{T-\Delta}$. In turn, 
The comparison principle
in Proposition \ref{schemecomparison} then implies
$$u_{\varepsilon}-u^{\Delta}\leq \sup_{\bar{Q}_{T}\backslash \bar{Q}_{T-\Delta}}(u_{\varepsilon}-u^{\Delta})^{+}+C(T-t+1)(\varepsilon+\Delta\varepsilon^{-3})\ \text{in}\ \bar{Q}_T.$$
Next, using the estimates (\ref{est2}) and (\ref{upperbound_2}), we
obtain $|u-u_{\varepsilon}|\leq C\varepsilon$ and, thus,
\begin{align*}
u-u^{\Delta}&=(u-u_{\varepsilon})+(u_{\varepsilon}-u^{\Delta})\\
&\leq C\varepsilon+\sup_{\bar{Q}_{T}\backslash
\bar{Q}_{T-\Delta}}(u_{\varepsilon}-u^{\Delta})^{+}+C(\varepsilon+\Delta\varepsilon^{-3})\\
&\leq \sup_{\bar{Q}_{T}\backslash
\bar{Q}_{T-\Delta}}(u-u^{\Delta})^{+}+C(\varepsilon+\Delta\varepsilon^{-3})\
\text{in}\ \bar{Q}_T.
\end{align*}
By choosing $\varepsilon=\Delta^{\frac14}$, we further obtain
$$u-u^{\Delta}\leq \sup_{\bar{Q}_{T}\backslash \bar{Q}_{T-\Delta}}(u-u^{\Delta})^{+}+C\Delta^{\frac14}\le C\Delta^{\frac14}\ \ \text{in}\ \bar{Q}_T,$$
where the last inequality follows from the estimate
(\ref{estimate_final_interval}) in Lemma \ref{errorsmall}.
\end{proof}

%%%%%%%%%%%%%%%%%%%%%%%%%%%%%%%%%%%%%%%%%%%%%%%%%%%%%%%%%%%%%%%%%%%%%%%%%%
\subsection{Lower bound for the approximation {error}}

To obtain a lower bound for the approximation error, we cannot follow the above
perturbation procedure to construct approximate smooth supersolutions {of} (\ref{PDE_1}). This is because if we perturb its coefficients in an opposite way to obtain a viscosity supersolution, its convolution
with the mollifier may no longer be a supersolution due to the
convexity of the function $g$ in (\ref{PDE_1}). One way to solve the convexity problem is to interchange the
roles of  equation (\ref{PDE_1})-(\ref{terminal}) and its approximation scheme
(\ref{semischeme}) (as in \cite{HL1} and \cite{Jakobsen}), and perturb the scheme instead of the equation. This approach, however, does not work either, as $u^{\Delta}$ does not in general have enough regularity (Assumption 2.5 in \cite{Jakobsen} fails).

To overcome these difficulties, in this paper we follow the idea of Barles and Jakobsen \cite{BJ} to
build approximate supersolutions which are smooth at the ``right
points'' by introducing an appropriate optimal switching and stopping (with possible stochastic control) systems. Compared to the case in \cite{BJ} where they use standard optimal switching systems to approximate an HJB equation, the systems herein are more difficult. This is because we need to add an extra region (equation) to the standard system in order to approximate a HJB equation with obstacle, i.e. a variational inequality. Thus, here we call the new system a \emph{variant switching system} and give more details in Section \ref{App2}.

To apply the above method to the problem herein, we first observe that, using the convex dual function $L$ introduced in (\ref{L}), we can write the equation (\ref{PDE_1}) as 
\begin{equation}\label{HJBeq}
\max\{-\partial_{t}u+\sup_{q\in\mathbb{R}^{n}}\mathcal{L}^{q}\left(t,x,\partial_{x}u,\partial_{xx}u\right),u-f(t,x)\}=0,
\end{equation}
with
\begin{equation*}
\mathcal{L}^{q}(t,x,p,X):=-\frac{1}{2}\text{Trace}\left(\sigma\sigma^T(t,x)X\right)-(b(t,x)-q)\cdot
p-L(t,x,q).
\end{equation*}
It then follows from Proposition 2.3 (iv) in \cite{Huang} and Proposition \ref{solutionproperty} that the
supremum in (\ref{HJBeq}) can be achieved at some point $q^{*}$ with
%$|q^{*}|\leq\xi(|\partial_{x}u|)$. Furthermore, Proposition
%\ref{solutionproperty} implies that 
$|q^{*}|\leq C$. 
%for some
%constant $C$ only depending on $M$ and $T$.
%This in turn suggests that $q^{*}$ is bounded uniformly by some constant $C$ only depending on $T$, $M$, and the Legendre transform $L$.
Thus, we rewrite the equation (\ref{HJBeq}) as
\begin{equation}\label{hjbeq111}
\max\{-\partial_{t}u+\sup_{q\in
K}\mathcal{L}^{q}\left(t,x,\partial_{x}u,\partial_{xx}u\right),u-f(t,x)\}=0,
\end{equation}
where $K\subset\mathbb{R}^{n}$ is a compact set. Since  $K$ is separable, it has a countable dense subset, say
$K_{\infty}=\{q_{1},q_{2},q_{3},...\}$ and, in turn, the continuity of
$\mathcal{L}^{q}$ in $q$ implies that
\begin{equation*}
\sup_{q\in K}\mathcal{L}^{q}(t,x,p,X)=\sup_{q\in
K_{\infty}}\mathcal{L}^{q}(t,x,p,X).
\end{equation*}
Therefore, (\ref{HJBeq}) further reduces to
\begin{equation}\label{HJBeq1}
\max\{-\partial_{t}u+\sup_{q\in
K_{\infty}}\mathcal{L}^{q}\left(t,x,\partial_{x}u,\partial_{xx}u\right),u-f(t,x)\}=0.
\end{equation}

Next, for $m\in\mathbb{Z}^{+}$, we consider the following equations to approximate the original equation (\ref{PDE_1})/(\ref{HJBeq1})-(\ref{terminal}),
\begin{align}
\max\{-\partial_{t}u^{m}+\sup_{q\in
{K}_{m}}\mathcal{L}^{q}\left(t,x,\partial_{x}u^{m},\partial_{xx}u^{m}\right),u^{m}-f(t,x)\}&=0 & \text{in}\ Q_T;\label{finiteHJB}\\
u^{m}(T,x)&=U(x) & \text{in}\ \mathbb{R}^n,\label{finiteHJB_terminal}
\end{align}
where ${K}_{m}:=\{q_{1},...,q_{m}\}\subset K_{\infty}$ consisting
the first $m$ points in $K_{\infty}$ and satisfying
$\cup_{m\ge 1}K_m=K_{\infty}$.
It then follows from standard optimal stopping control problem (see \cite{BT, Fleming}) that
(\ref{finiteHJB})-(\ref{finiteHJB_terminal}) admits a unique
viscosity solution $u^m\in\mathcal{C}_b^1(\bar{Q}_T)$, with
$|u^m|_{1}\leq C$ independent of $m$.
%for some constant $C$ only depending on $M$,
%$T$ and the compact set $K$. 
Then, Arzela-Ascoli's theorem yields that there exists a subsequence of
$\{u^{m}\}$, denoted as $\{u^{m_{n}}\}$, such that, as $n\rightarrow\infty$.
\begin{equation}\label{lemma1}
u^{m_{n}}\rightarrow u\ \text{locally uniformly\ in}\ \bar{Q}_{T},
\end{equation}
and moreover, since the terminal condition for $u$ and that for $u^{m_{n}}$ ((\ref{terminal}) and (\ref{finiteHJB_terminal}) respectively) coincide, it follows from their regularity result that, for $n\in\mathbb{Z}^{+}$
\begin{equation}\label{lemma2}
|u(t,\cdot)-u^{m_{n}}(t,\cdot)|_{0}\le |u(t,\cdot)-u(T,\cdot)|_{0}+|u^{m_{n}}(t,\cdot)-u^{m_{n}}(T,\cdot)|_{0}\le C\sqrt{T-t}.
\end{equation}

%\begin{lemma}\label{lemma1}
%Suppose that Assumption \ref{data assumption} is satisfied. Then,
%there exists a unique viscosity solution $u^m$, with
%$u^m\in\mathcal{C}_b^1(\bar{Q}_T)$, of the HJB equation
%(\ref{switching})-(\ref{switching_terminal}), with $|u^m|_{1}\leq C$ only depending on $M$ and
%$T$. Moreover, $u^m\rightarrow u$ uniformly in $(t,x)\in\bar{Q}_{T}$
%as $m\rightarrow\infty$.
%\end{lemma}
%
%\begin{proof}
%We first show that we can indeed use $u^{m}$ in (\ref{finiteHJB}) to
%approximate $u$. This is true since by standard regularity results
%for HJB equations (see Proposition 2.1 of \cite{BJ}), we have
%$$|u^{m}|_{1}\leq C$$
%where the constant $C$ only depends on $M$ and $T$. Then by
%Arzela-Ascoli's theorem, there exists a subsequence of
%$\{u^{m}\}_{m}$, also denoted by $\{u^{m}\}_{m}$, converges locally
%uniformly. Obviously, this limit function satisfies
%$$
%-\partial_{t}u(t,x)+\sup_{q\in
%K_{\infty}}\mathcal{L}^{q}(t,x,\partial_{x}u,\partial_{xx}u)(t,x)=0,
%$$
%\end{proof}
For $m\in\mathbb{Z}^{+}$, we then construct a sequence of (local) smooth supersolutions to
approximate the solution $u^m$ of (\ref{finiteHJB})-(\ref{finiteHJB_terminal}). To this end, we consider the following $m+1$ dimensional variant  switching system:
\begin{eqnarray}\label{switching}
\left\{
\begin{array}{l}
 \max\left\{-\partial_tv_i+\mathcal{L}^{q_{i}}(t,x,\partial_{x}v_i,\partial_{xx}v_i),v_i-\mathcal{M}^{k}_{i}v\right\}=0, \ \ \ i\in\mathcal{I}:=\{1,...,m\},\\
 \max\{v_{m+1}-f,v_{m+1}-\mathcal{M}_{m+1}^{k}v\}=0,
\end{array}
\ \ \ \text{in}\ Q_{T},\right.
\end{eqnarray}
with 
\begin{equation}\label{switching_terminal}
v_i(T,x)=U(x), \ \ \ i\in\bar{\mathcal{I}}:=\mathcal{I}\cup\{m+1\} 
\end{equation}
where
$\mathcal{M}^{k}_{i}v:=\min_{j\neq i,\ j\in\bar{\mathcal{I}}}\{v_{j}+k\}$,
for some constant $k>0$ representing the switching cost.

By the results from the next section, we then have the following existence, uniqueness and regularity results for the
solution $v$ of variant switching system
(\ref{switching})-(\ref{switching_terminal}), and a comparison
between $v$ and $u^m$.

\begin{proposition}\label{proposition}
Suppose that Assumption \ref{data assumption} is satisfied. Then, for any $m\in\mathbb{Z}^{+}$,
there exists a unique viscosity solution $v=(v_1,\dots,v_{m+1})$ of (\ref{switching})-(\ref{switching_terminal}), with $|v|_{1}\leq C$.
%for some constant $C$ depending only on $M$ and $T$. 
Moreover, for $k$ small enough,
\begin{equation}\label{optimal_switching_approx}
0\leq v_{i}-u^m\leq Ck^{\frac13}\ \ \text{in}\ \bar{Q}_{T},\ \ i\in\bar{\mathcal{I}}.
\end{equation}
\end{proposition}

\begin{proof}
The existence, uniqueness and regularity result of the viscosity solution $v$ is given by Theorem \ref{switchingthm} in the next section. To get the estimates (\ref{optimal_switching_approx}), we first check that $w=(u^{m},...,u^{m})$ is a subsolution of (\ref{switching}), then comparison result for the variant switching system (which is implied by Theorem \ref{switchingcd}) yields $u^{m}\leq v_{i}$ for $i\in\bar{\mathcal{I}}$.

To derive the other bound, we follow the same regularization procedure as that in Section \ref{upbd}. For small enough $\varepsilon>0$, we consider the following perturbed system of (\ref{switching}):
\begin{eqnarray}\label{Ptbedswitching}
\left\{
\begin{array}{l}
 \max\displaystyle\left\{-\partial_tv^{\varepsilon}_i+\sup_{(\tau,e)\in\Theta^{\varepsilon}}\mathcal{L}^{q_{i}}(t+\tau,x+e,\partial_{x}v^{\varepsilon}_i,\partial_{xx}v^{\varepsilon}_i),v^{\varepsilon}_i-\mathcal{M}^{k}_{i}v^{\varepsilon}\right\}=0, \ \ i\in\mathcal{I},\\
\displaystyle \max\left\{v^{\varepsilon}_{m+1}-f,v^{\varepsilon}_{m+1}-\mathcal{M}_{m+1}^{k}v^{\varepsilon}\right\}=0,
\end{array}
\ \text{in}\ Q_{T+\varepsilon^2},\right.
\end{eqnarray}
with 
\begin{equation}\label{Ptbedswitching_terminal}
v^{\varepsilon}_i(T+\varepsilon^{2},x)=U(x), \ \ \ i\in\bar{\mathcal{I}} 
\end{equation}
where $\Theta^{\varepsilon}=[-\varepsilon^{2},0]\times \varepsilon B(0,1)$. Note that here we also extend the coefficients $\sigma$, $b$, $f$ and $L$ appropriately.
%Next, by applying again Theorem \ref{switchingthm} as well as Theorem \ref{switchingcd}, we immediately get the following existence and regularity result for the perturbed switching system:
%\begin{lemma}\label{Ptbedswitchingproperty}
%Under Assumption \ref{data assumption}, then there exists a unique solution $v^{\varepsilon}=(v^{\varepsilon}_{1},...,v^{\varepsilon}_{m+1})$ of (\ref{Ptbedswitching})-(\ref{Ptbedswitching_terminal}) satisfying $|v^{\varepsilon}|_{1}\leq C$ and $|v^{\varepsilon}-v|_{0}\leq C\varepsilon$, where $v$ solves (\ref{switching})-(\ref{switching_terminal}) and the constant $C$ only depends on $M$ and $T$.
%\end{lemma}
It then follows from Theorem \ref{switchingthm} and Theorem \ref{switchingcd} that
(\ref{Ptbedswitching})-{(\ref{Ptbedswitching_terminal})} admits a
unique viscosity solution
$v^{\varepsilon}=(v_1^{\varepsilon},\dots,v_{m+1}^{\varepsilon})$ such that for $i\in\bar{\mathcal{I}}$, 
\begin{equation}\label{optimal_switching_approx_111}
|v_{i}^{\varepsilon}|_{1}\leq C \ \text{and} \ \ {|v^{\varepsilon}_i-v_i|}\leq C\varepsilon\ \text{in}\ \bar{Q}_T.
\end{equation}
It then follows from (\ref{Ptbedswitching}) that for any $(\tau,e)\in\Theta^{\varepsilon}$,
$$-\partial_{t}v_{i}^{\varepsilon}+\mathcal{L}^{q_{i}}(t+\tau,x+e,\partial_{x}v_{i}^{\varepsilon},\partial_{xx}v_{i}^{\varepsilon})\leq0 \ \ \ \text{in} \ \ Q_{T+\varepsilon^{2}}, \ \ \  i\in\mathcal{I}.
$$
%By change of variable, for any $\theta\in\Theta^{\varepsilon}$, $v_{i}^{\varepsilon}(t-\tau,x-e)$ is a subsolution of
%$$-\partial_{t}w_{i}+\mathcal{L}^{q_{i}}(t,x,\partial_{x}w_{i},\partial_{xx}w_{i})=0 \ \ \ \text{in} \ \ Q_{T}, \ \ \  i\in\mathcal{I}.
%$$
For each $i\in\mathcal{I}$, define $v_{i,\varepsilon}:=v_{i}^{\varepsilon}*\rho_{\varepsilon}$, where $\rho_{\varepsilon}$ is defined in (\ref{mollifer}). By the same argument as in Section \ref{upbd}, we have
\begin{equation}\label{sub222}
-\partial_{t}v_{i,\varepsilon}+\mathcal{L}^{q_{i}}(t,x,\partial_{x}v_{i,\varepsilon},\partial_{xx}v_{i,\varepsilon})\leq0 \ \ \ \text{in} \ \ Q_{T}, \ \ \  i\in\mathcal{I}.
\end{equation}

On the other hand, it follows again from (\ref{Ptbedswitching}) that for any $i\in\bar{\mathcal{I}}$, $v_{i}^{\varepsilon}\leq\mathcal{M}_{i}^{k}v^{\varepsilon}=\min_{j\neq i,j\in\bar{\mathcal{I}}}v_{j}^{\varepsilon}+k$ in $Q_{T+\varepsilon^{2}}$, which implies that 
$$|v_{i}^{\varepsilon}-v_{j}^{\varepsilon}|_{0}\leq k,\ \ \ i,j\in\bar{\mathcal{I}}.$$
Then by standard properties of mollifiers, we have
$$|\partial_{t}v_{i,\varepsilon}-\partial_{t}v_{j,\varepsilon}|_0\leq Ck\varepsilon^{-2};\ \ \
|D_{x}^nv_{i,\varepsilon}-D_{x}^nv_{j,\varepsilon}|_0\leq Ck\varepsilon^{-n}, \ \ \ n\in\mathbb{N}, \ \ \ i,j\in\mathcal{I},$$
%where the constant $C$ only depends on the mollifier $\rho$. 
These estimates then yield that
$$|-\partial_{t}v_{j,\varepsilon}+\mathcal{L}^{q_{j}}(t,x,\partial_{x}v_{j,\varepsilon},\partial_{xx}v_{j,\varepsilon})+\partial_{t}v_{i,\varepsilon}-\mathcal{L}^{q_{j}}(t,x,\partial_{x}v_{i,\varepsilon},\partial_{xx}v_{i,\varepsilon})|\leq Ck(\varepsilon^{-2}+\varepsilon^{-1}) \ \ \ \text{in} \ \ Q_{T}, \ \ \  i,j\in\mathcal{I}.
$$
%where the constant $C$ depends on $M$(and $\rho$). 
This together with (\ref{sub222}) gives that
$$-\partial_{t}v_{i,\varepsilon}+\mathcal{L}^{q_{j}}(t,x,\partial_{x}v_{i,\varepsilon},\partial_{xx}v_{i,\varepsilon})\leq Ck(\varepsilon^{-2}+\varepsilon^{-1}) \ \ \ \text{in} \ \ Q_{T}, \ \ \  i,j\in\mathcal{I},$$
which means
$$-\partial_{t}v_{i,\varepsilon}+\sup_{q\in {K}_{m}}\mathcal{L}^{q}(t,x,\partial_{x}v_{i,\varepsilon},\partial_{xx}v_{i,\varepsilon})\leq Ck(\varepsilon^{-2}+\varepsilon^{-1}) \ \ \ \text{in} \ \ Q_{T}, \ \ \  i\in\mathcal{I}.$$
Next, since for any $i\in\mathcal{I}$, $v_{i}^{\varepsilon}\le\mathcal{M}_{i}^{k}v^{\varepsilon}\le v^{\varepsilon}_{m+1}+k\le f+k$ in $Q_{T+\varepsilon^{2}}$, we obtain that for $(t,x)\in Q_{T}$,
$$v_{i,\varepsilon}(t,x)-f(t,x)\le k+\int_{-\varepsilon^2< \tau< 0}\int_{|e|<
\varepsilon}(f(t-\tau,x-e)-f(t,x))\rho_{\varepsilon}(\tau,e)ded\tau\le k+ M\varepsilon.$$

From the above two inequalities we can see that for any $i\in\mathcal{I}$, $v_{i,\varepsilon}-(T-t)Ck(\varepsilon^{-2}+\varepsilon^{-1})-k-M\varepsilon$ is a subsolution of (\ref{finiteHJB}) in $Q_{T}$ with terminal value $v_{i,\varepsilon}(T,\cdot)-k-M\varepsilon$. Then by standard continuous dependence result for (\ref{finiteHJB}), we have for $i\in\mathcal{I}$,
\begin{align*}
v_{i,\varepsilon}-u^{m} \leq &\ |(v_{i,\varepsilon}(T,\cdot)-k-M\varepsilon-u^{m}(T,\cdot))^{+}|_{0}+(T-t)Ck(\varepsilon^{-2}+\varepsilon^{-1})+k+M\varepsilon\\
 \le &\
 |v_{i,\varepsilon}(T,\cdot)-v_{i}^{\varepsilon}(T+\varepsilon^{2},\cdot)|_{0}+C(\varepsilon+k\varepsilon^{-2}+k\varepsilon^{-1}+k) \\
 \le &\
 |v_{i,\varepsilon}(T,\cdot)-v_{i}^{\varepsilon}(T+\varepsilon^{2},\cdot)|_{0}+C(\varepsilon+k\varepsilon^{-2}) \ \ \ \text{in} \ \ Q_{T}.
\end{align*}
Hence by the properties of mollifiers and regularity of $v^{\varepsilon}$, we have for $i\in\mathcal{I}$,\begin{align*}
v^{\varepsilon}_{i}-u^{m} = &\ v_{i}^{\varepsilon}-v_{i,\varepsilon}+v_{i,\varepsilon}-u^{m} \\
 \leq &\
 C\varepsilon+|v_{i,\varepsilon}(T,\cdot)-v_{i}^{\varepsilon}(T,\cdot)|_{0}\\
 + &\
 |v_{i}^{\varepsilon}(T,\cdot)-v_{i}^{\varepsilon}(T+\varepsilon^{2},\cdot)|_{0}+C(\varepsilon+k\varepsilon^{-2})\\
 \leq &\ C(\varepsilon+k\varepsilon^{-2}) \ \ \ \text{in} \ \ Q_{T}.
\end{align*}
%where the contant $C$ only depends on $M$ and $T$. 
Furthermore, since $v^{\varepsilon}_{m+1}\le\min_{i\in\mathcal{I}}v^{\varepsilon}_{i}+k$, the above inequality  holds for all $i\in\bar{\mathcal{I}}$. It finally follows from (\ref{optimal_switching_approx_111}) that for $i\in\bar{\mathcal{I}}$,
\begin{align*}
v_{i}-u^{m} = &\ v_{i}-v_{i}^{\varepsilon}+v_{i}^{\varepsilon}-u^{m} \\
 \leq &\
 C\varepsilon+C(\varepsilon+k\varepsilon^{-2})\\
 \leq &\ C(\varepsilon+k\varepsilon^{-2}) \ \ \ \text{in} \ \ Q_{T}.
\end{align*}
We then choose $\varepsilon=k^{\frac13}$ and finish the proof.
\end{proof}\\

%In fact, as $k\to0$, every component of $v$ converges locally
%uniformly to the solution of the following HJB equation:
%\begin{equation}\label{HJBeq2}
%-\partial_{t}u+\sup_{q\in
%\mathcal{A}}\mathcal{L}^{q}(t,x,\partial_{x}u,\partial_{xx}u)=0\
%\text{in}\ Q_{T}
%\end{equation}
%with terminal condition $u(T,x)=U(x)$, where
%$\mathcal{A}=\cup_{i=1}^{m}\mathcal{A}_{i}$. Note that, by Remark
%\ref{transform}, if we set $\mathcal{A}=K$, the solution $u$ is
%exactly the solution of our original PDE (\ref{PDE_1}). More
%specificly, we have the following error estimate result:
%\begin{theorem}\label{errorboundswitching}
%Under Assumption \ref{data assumption}, if
%$\mathcal{A}=\cup_{i=1}^{m}\mathcal{A}_{i}\subset\mathbb{R}^{n}$ is
%a compact set, and let $u$ and $v$ be the solutions of
%(\ref{HJBeq2}) with (\ref{terminal}) and (\ref{switching})-(\ref{switching_terminal})
%respectively, then for any $k>0$, there exists a constant $C$ only
%depending on $T$, $M$ and $\mathcal{A}$ such that
%$$0\leq v_{i}-u\leq C(k^{\frac13}+k^{\frac23}), \ \ \text{in}\ Q_{T}, \ \ \ i\in\mathcal{I}.$$
%\end{theorem}

%The proof essentially follows from Proposition 2.1 and Theorem 2.3
%of \cite{BJ}, so it is omitted. We only remark that since we do not
%require the switching cost $k\leq 1$ at this stage, we keep the term
%$k^{\frac23}$ in the above estimate. Note that it will not affect
%the convergence rate of the splitting scheme.
%
Next, still following the approach of \cite{BJ}, we construct smooth
approximations of $v_i$. Since when $i\in\mathcal{I}$, in the continuation region of
(\ref{switching}), the solution $v_i$ satisfies the \emph{linear}
equation, namely,
$$-\partial_tv_i+\mathcal{L}^{q_{i}}(t,x,\partial_{x}v_i,\partial_{xx}v_i)=0\ \ \text{in}\ \{(t,x)\in{Q}_T:v_i(t,x)<\mathcal{M}_i^{k}v(t,x)\}, \ \ i\in\mathcal{I},$$
we may perturb its coefficients to obtain a sequence of smooth
supersolutions. This will in turn give a lower bound of the error
$v_{i}-u^{\Delta}$, and thus a lower bound of $u^{m}-u^{\Delta}$ by the estimate (\ref{optimal_switching_approx}). A subtle point herein is how to identify the
continuation region by appropriately choosing the switching cost
$k$. For this, we follow the idea used in Lemma 3.4 of \cite{BJ}.

\begin{proposition}\label{lowbd1}
Suppose that Assumption \ref{data assumption} is satisfied. Let $\Delta\in(0,T)$, $m\in\mathbb{Z}^{+}$, 
$u^{\Delta}\in\mathcal{C}_b(\bar{Q}_{T})$ be the unique solution of the approximation scheme
(\ref{semischeme}) and $u^m\in\mathcal{C}_b^{1}(\bar{Q}_{T})$ be the unique viscosity solution of equation 
(\ref{finiteHJB})-(\ref{finiteHJB_terminal}). Then, 
%there exists a
%constant $C$ only depending on $M$ and $T$ such that
$$u^{\Delta}-u^{m}\leq \sup_{\bar{Q}_{T}\backslash\bar{Q}_{T-\Delta}}(u^{\Delta}-u^{m})^{+}+C\Delta^{\frac{1}{10}}\ \text{in}\ \bar{Q}_{T}.$$
\end{proposition}

\begin{proof}
In analogy to (\ref{Ptbedswitching}) but in the opposite direction, for small enough $\varepsilon> 0$, we perturb the coefficients of the system
(\ref{switching})-(\ref{switching_terminal}), and consider the
following variant switching system:
\begin{eqnarray}\label{Ptbedswitching1}
\left\{
\begin{array}{l}
 \max\displaystyle\left\{-\partial_tv^{\varepsilon}_i+\inf_{(\tau,e)\in\Theta^{\varepsilon}}\mathcal{L}^{q_{i}}(t+\tau,x+e,\partial_{x}v^{\varepsilon}_i,\partial_{xx}v^{\varepsilon}_i),v^{\varepsilon}_i-\mathcal{M}^{k}_{i}v^{\varepsilon}\right\}=0, \ \ i\in\mathcal{I},\\
\displaystyle \max\left\{v^{\varepsilon}_{m+1}-f,v^{\varepsilon}_{m+1}-\mathcal{M}_{m+1}^{k}v^{\varepsilon}\right\}=0,
\end{array}
\ \text{in}\ Q_{T+\varepsilon^2},\right.
\end{eqnarray}
with 
\begin{equation}\label{Ptbedswitching1_terminal}
v_i^{\varepsilon}(T+\varepsilon^{2},x)=U(x),\ \ \ i\in\bar{\mathcal{I}},
\end{equation}
%where, $\Theta^{\varepsilon}=[-\varepsilon^{2},0]\times \varepsilon B(0,1)$,
%and the coefficients $\sigma$, $b$, $f$ and $L$ are appropriately
%extended such that Assumption \ref{data assumption} still holds. 
It then follows again from Theorem \ref{switchingthm} and Theorem \ref{switchingcd} that
(\ref{Ptbedswitching1})-{(\ref{Ptbedswitching1_terminal})} admits a
unique viscosity solution
$v^{\varepsilon}=(v_1^{\varepsilon},\dots,v_{m+1}^{\varepsilon})$, with
$|v_{i}^{\varepsilon}|_{1}\leq C$ and, moreover, for each $i\in\bar{\mathcal{I}}$,
%\begin{equation}\label{optimal_switching_approx_1}
$${|v^{\varepsilon}_i-v_i|}\leq C\varepsilon\ \text{in}\ \bar{Q}_T.$$
%\end{equation} 
%, where the constant $C$ depends only on
%$M$ and $T$.
%the existence and uniqueness of the (\ref{Ptbedswitching1}-\ref{Ptbedswitching1_terminal})
%\begin{lemma}\label{Ptbedswitchingproperty1}
%Under Assumption \ref{data assumption}, if $\mathcal{A}=\cup_{i=1}^{m}\mathcal{A}_{i}\subset\mathbb{R}^{n}$
%is a compact set, then there exists a unique solution $v^{\varepsilon}$ of (\ref{Ptbedswitching1}-\ref{Ptbedswitching1_terminal}) satisfying $|v^{\varepsilon}|_{1}\leq C$ and $|v^{\varepsilon}-v|_{0}\leq C\varepsilon$, where $v$ solves (\ref{switching})-(\ref{switching_terminal}) and the constant $C$ only depends on $T$, $M$ and $\mathcal{A}$.
%\end{lemma}
%Now set $\mathcal{A}_{i}=\{q_{i}\}$ for $i\in\mathcal{I}$ and hence $\cup_{i=1}^{m}\mathcal{A}_{i}=\mathcal{A}^{m}$ is compact. Note that under this setting, in the definition of $F_{i}^{k,\varepsilon}$ we can move out the ``sup'' operator and replace $\mathcal{L}^{q}$ by $\mathcal{L}^{q_{i}}$ since $q$ must be $q_{i}$. We then apply Theorem \ref{errorboundswitching} to get
%$$0\leq v_{i}-u^{m}\leq C(k^{\frac13}+k^{\frac23})\ \text{in} \ Q_{T}, \ \ \ i\in\mathcal{I},$$
%where $v$ is the solution of (\ref{switching})-(\ref{switching_terminal}) under our setting. Together with Lemma \ref{Ptbedswitchingproperty1} we have
In turn, this and inequality (\ref{optimal_switching_approx}) imply that, for each
$i\in\bar{\mathcal{I}}$,
\begin{equation}\label{Ptbederror}
{|v_{i}^{\varepsilon}- u^{m}|}\leq
{|v_{i}^{\varepsilon}-v_i|}+{|v_{i}-u^m|}\leq
C(\varepsilon+k^{\frac13})\ \text{in}\ \bar{Q}_T.
\end{equation}

Next, we regularize $v^{\varepsilon}$ and define
$v_{i,\varepsilon}:=v_{i}^{\varepsilon}*\rho_{\varepsilon}$ in $\bar{Q}_{T}$ for
$i\in\bar{\mathcal{I}}$, where $\rho_{\varepsilon}$ is the mollifer defined
in (\ref{mollifer}). Then, $v_{i,\varepsilon}\in\mathcal{C}_{b}^{\infty}(\bar{Q}_{T})$,
\begin{equation}\label{convoerror}
|v_{i,\varepsilon}-v_{i}^{\varepsilon}|_{0}\leq C\varepsilon,
\end{equation}
and moreover, for positive integer $m$ and multiindex $n$,
\begin{equation}\label{convoproperty}
{|\partial_{t}^{m}D_{x}^{n}v_{i,\varepsilon}|_0\leq
C\varepsilon^{1-2m-|n|}}.
\end{equation}

Define $w_{\varepsilon}:=\min_{i\in\bar{\mathcal{I}}}v_{i,\varepsilon}$ in $\bar{Q}_{T}$. Then the function $w_{\varepsilon}\in\mathcal{C}_{b}(\bar{Q}_{T})$ and is
smooth in $\bar{Q}_T$ except for finitely many points. Then,
(\ref{Ptbederror}) and (\ref{convoerror}) yield
\begin{equation}\label{errtogether}
{|u^{m}-w_{\varepsilon}|}\leq C(\varepsilon+k^{\frac13})\
\text{in}\ \bar{Q}_T.
\end{equation}

Now we fix any $(t,x)\in\bar{Q}_{T-\Delta}$ and let
$j=\arg\min_{i\in\bar{\mathcal{I}}}v_{i,\varepsilon}(t,x)$. Then we obtain that
%Later on, we write $j$ instead of $j(t,x)$ for simplicity. Note that
%$w_{\varepsilon}(t,x)=v_{j,\varepsilon}(t,x)$, and for such $j$,
$$v_{j,\varepsilon}(t,x)-\mathcal{M}_{j}^{k}v_{\varepsilon}(t,x)=\max_{i\neq j, i\in\bar{\mathcal{I}}}\{v_{j,\varepsilon}(t,x)-v_{i,\varepsilon}(t,x)-k\}\leq
-k.$$ In turn, inequality (\ref{convoerror}) implies that
$$v_{j}^{\varepsilon}(t,x)-\mathcal{M}_{j}^{k}v^{\varepsilon}(t,x)\leq v_{j,\varepsilon}(t,x)-\mathcal{M}_{j}^{k}v_{\varepsilon}(t,x)+C\varepsilon\leq -k+C\varepsilon.$$
Furthermore, since $|v_{j}^{\varepsilon}|_1\leq C$, we also have for any $(\tau,e)\in \Theta^{\varepsilon}$,
\begin{align*}
v_{j}^{\varepsilon}(t-\tau,x-e)-\mathcal{M}_{j}^{k}v^{\varepsilon}(t-\tau,x-e)&\leq
v_{j}^{\varepsilon}(t,x)-\mathcal{M}_{j}^{k}v^{\varepsilon}(t,x)+C(|\tau|^{\frac12}+|e|)\\
&\leq -k+C\varepsilon+2C\varepsilon.
\end{align*}
By choosing $k= 4C\varepsilon$, we then obtain that, for any $(\tau,e)\in
\Theta^{\varepsilon}$,
%\begin{equation}\label{ctsregion}
%v_{j}^{\varepsilon}(t-\tau,x-e)-\mathcal{M}_{j}^{k}v^{\varepsilon}(t-\tau,x-e)<0.
%\end{equation}
$$v_{j}^{\varepsilon}(t-\tau,x-e)-\mathcal{M}_{j}^{k}v^{\varepsilon}(t-\tau,x-e)<0.
$$
%This means $(t-\tau,x-\tau)$, for $\theta\in \Theta^{\varepsilon}$,
%is in the continuation region of (\ref{Ptbedswitching1}). 
Therefore, the points $(t-\tau,x-e)$, for $(\tau,e)\in
\Theta^{\varepsilon}$, are in the continuation region of
(\ref{Ptbedswitching1}). Now we consider two cases when $j\in\mathcal{I}$ and when $j=m+1$ separately.

(i) If $j\in\mathcal{I}$, then we have that, for $(\tau,e)\in\Theta^{\varepsilon}$,
$$-\partial_{t}v_{j}^{\varepsilon}(t-\tau,x-e)+\inf_{(\tau',e')\in\Theta^{\varepsilon}}\mathcal{L}^{q_{j}}\left(t-\tau+\tau',x-e+e',\partial_{x}v_{j}^{\varepsilon}(t-\tau,x-e),\partial_{xx}v_{j}^{\varepsilon}(t-\tau,x-e)\right)=0,$$
and therefore,
$$-\partial_{t}v_{j}^{\varepsilon}(t-\tau,x-e)+\mathcal{L}^{q_{j}}\left(t,x,\partial_{x}v_{j}^{\varepsilon}(t-\tau,x-e),\partial_{xx}v_{j}^{\varepsilon}(t-\tau,x-e)\right)\ge0.$$
Using the definition of $v_{j,\varepsilon}$ and that
$\mathcal{L}^{q_{j}}$ is linear in $\partial_{x}v_{j}^{\varepsilon}$
and $\partial_{xx}v_{j}^{\varepsilon}$, we further have
\begin{align}\label{supsolution}
&-\partial_{t}v_{j,\varepsilon}(t,x)+\mathcal{L}^{q_{j}}\left(t,x,\partial_{x}v_{j,\varepsilon}(t,x),\partial_{xx}v_{j,\varepsilon}(t,x)\right)\notag\\
=&\ \int_{-\varepsilon^2< \tau< 0}\int_{|e|< \varepsilon}
\left(-\partial_{t}v_{j}^{\varepsilon}(t-\tau,x-e)+\mathcal{L}^{q_{j}}\left(t,x,\partial_{x}v_{j}^{\varepsilon}(t-\tau,x-e),\partial_{xx}v_{j}^{\varepsilon}(t-\tau,x-e)\right)\right)\rho_{\varepsilon}(\tau,e)ded\tau\notag\\
\ge&\ 0.
\end{align}
Since $w_{\varepsilon}(t,x)=v_{j,\varepsilon}(t,x)$ and $w_{\varepsilon}(t+\Delta,\cdot)\leq v_{j,\varepsilon}(t+\Delta,\cdot)$, we apply Proposition \ref{scheme property} and use the
estimate (\ref{convoproperty}) and (\ref{supsolution}) to obtain that
%$(t,x)\in\bar{Q}_{T-\Delta}$ such that $j(t,x)\in\mathcal{I}$,
\begin{align*}
 &\ \bar{S}(\Delta,t,x,w_{\varepsilon}(t,x),w_{\varepsilon}(t+\Delta,\cdot))\\
 \ge &\
 \bar{S}(\Delta,t,x,v_{j,\varepsilon}(t,x),v_{j,\varepsilon}(t+\Delta,\cdot))\\
 \ge &\
 -\partial_{t}v_{j,\varepsilon}(t,x)+g(t,x,\partial_{x}v_{j,\varepsilon}(t,x),\partial_{xx}v_{j,\varepsilon}(t,x))-C\Delta\varepsilon^{-3}\\
 = &\
 -\partial_{t}v_{j,\varepsilon}(t,x)+\sup_{q\in\mathbb{R}^{n}}\mathcal{L}^{q}(t,x,\partial_{x}v_{j,\varepsilon}(t,x),\partial_{xx}v_{j,\varepsilon}(t,x))-C\Delta\varepsilon^{-3}\\
 \ge &\
  -\partial_{t}v_{j,\varepsilon}(t,x)+\mathcal{L}^{q_{j}}(t,x,\partial_{x}v_{j,\varepsilon}(t,x),\partial_{xx}v_{j,\varepsilon}(t,x))-C\Delta\varepsilon^{-3}\\
  \ge &\
  -C\Delta\varepsilon^{-3}.
\end{align*}
%for some constant $C$ only depends on $M$ and $T$, where we used
%(\ref{supsolution}) in the last inequality.

(ii) If $j=m+1$, then we have that, for $(\tau,e)\in\Theta^{\varepsilon}$,
$$v^{\varepsilon}_{j}(t-\tau,x-e)-f(t-\tau,x-e)=0.$$
%Thus, it follows obviously that, for any
%$(t,x)\in\bar{Q}_{T-\Delta}$ such that $j(t,x)=m+1$,
Since $w_{\varepsilon}(t,x)=v_{j,\varepsilon}(t,x)$, we then obtain that
\begin{align*}
 &\ \bar{S}(\Delta,t,x,w_{\varepsilon}(t,x),w_{\varepsilon}(t+\Delta,\cdot)) \ge w_{\varepsilon}(t,x)-f(t,x)\\
 = &\
 \int_{-\varepsilon^2< \tau< 0}\int_{|e|<
\varepsilon}(f(t-\tau,x-e)-f(t,x))\rho_{\varepsilon}(\tau,e)ded\tau\\
  \ge &\
  -M\varepsilon.
\end{align*}
%$$w_{\varepsilon}(t,x)-f(t,x)=\int_{-\varepsilon^2< \tau< 0}\int_{|e|<
%\varepsilon}(f(t-\tau,x-e)-f(t,x))\rho_{\varepsilon}(\tau,e)ded\tau\ge  -M\varepsilon.$$

%By combining the above two cases, we obtain that, for any
%$(t,x)\in\bar{Q}_{T-\Delta}$,
Thus, in any case, we always have
$$\bar{S}(\Delta,t,x,w_{\varepsilon}(t,x),w_{\varepsilon}(t+\Delta,\cdot))\ge -C(\Delta\varepsilon^{-3}+\varepsilon).$$
Note that the right hand side of the above inequality does not depend on $(t,x)$, thus this inequality holds in $\bar{Q}_{T-\Delta}$.
In turn, the comparison principle in Proposition
\ref{schemecomparison} implies that
$$u^{\Delta}-w_{\varepsilon}\leq \sup_{\bar{Q}_{T}\backslash \bar{Q}_{T-\Delta}}(u^{\Delta}-w_{\varepsilon})^{+}+C(T-t+1)(\Delta\varepsilon^{-3}+\varepsilon)\ \text{in}\ \bar{Q}_T.$$
Combining the above inequality with (\ref{errtogether}), we further
get
\begin{align*}
u^{\Delta}-u^{m} = &\ (u^{\Delta}-w_{\varepsilon})+(w_{\varepsilon}-u^{m})\\
 \leq &\
 \sup_{\bar{Q}_{T}\backslash
 \bar{Q}_{T-\Delta}}(u^{\Delta}-w_{\varepsilon})^{+}+C(T-t+1)(\Delta\varepsilon^{-3}+\varepsilon)+
 C(\varepsilon+k^{\frac13})\\
 \leq &\
 \sup_{\bar{Q}_{T}\backslash\bar{Q}_{T-\Delta}}(u^{\Delta}-u^{m})^{+}+C(\varepsilon^{\frac13}+\Delta\varepsilon^{-3})\\
 \leq &\
 \sup_{\bar{Q}_{T}\backslash\bar{Q}_{T-\Delta}}(u^{\Delta}-u^{m})^{+}+C\Delta^{\frac{1}{10}}\
 \text{in}\ \bar{Q}_T,
\end{align*}
where we {used} $k=4C\varepsilon$ in the second to last inequality, and
{choose} $\varepsilon=\Delta^{\frac{3}{10}}$ in the last
inequality.
\end{proof}\\

We are now ready to obtain the lower bound for the  approximation error.
\begin{theorem}\label{theorem_error_2}
Suppose that Assumption \ref{data assumption} is satisfied. Let $\Delta\in(0,T)$,
$u^{\Delta}\in\mathcal{C}_b(\bar{Q}_{T})$ be the unique solution of the approximation scheme
(\ref{semischeme}) and $u\in\mathcal{C}_b^{1}(\bar{Q}_{T})$ be the unique viscosity solution of equation (\ref{PDE_1})-(\ref{terminal}). Then, 
$$u-u^{\Delta}\ge -C\Delta^{\frac{1}{10}}\ \text{in}\ \bar{Q}_T.$$
\end{theorem}

\begin{proof}
Applying Proposition \ref{lowbd1} to the sequence $\{u^{m_{n}}\}$, we get for $n\in\mathbb{Z}^{+}$,
\begin{align*}
u^{\Delta}-u = &\ (u^{\Delta}-u^{m_{n}})+(u^{m_{n}}-u)\\
 \leq &\
 \sup_{\bar{Q}_{T}\backslash\bar{Q}_{T-\Delta}}(u^{\Delta}-u^{m_{n}})^{+}+C\Delta^{\frac{1}{10}}+(u^{m_{n}}-u)\\
 \leq &\
 \sup_{\bar{Q}_{T}\backslash
 \bar{Q}_{T-\Delta}}(u^{\Delta}-u)^{+}+
 \sup_{\bar{Q}_{T}\backslash \bar{Q}_{T-\Delta}}(u-u^{m_{n}})^{+}+C\Delta^{\frac{1}{10}}+u^{m_{n}}-u\\
 \leq &\
 C\Delta^{\frac{1}{10}}+u^{m_{n}}-u \ \ \ \text{in} \ \ \bar{Q}_{T},
\end{align*}
where we applied estimate (\ref{estimate_final_interval}) in Lemma
\ref{errorsmall} and estimate (\ref{lemma2}) to the last inequality. The conclusion then follows
by letting $n\to\infty$ and using (\ref{lemma1}).
\end{proof}

\section{Well-posedness, Regularity, and Continuous Dependence for a Variant Switching System}\label{App2}
In this section we construct the well-posedness, regularity, and continuous dependence results for general variant switching systems, which include equation (\ref{switching}), (\ref{Ptbedswitching}) and (\ref{Ptbedswitching1}) as special cases. These results are vitally important in the previous section when obtaining the lower bound.

For $T>0$ and $m\in\mathbb{Z}^{+}$, we consider the following $m+1$ dimensional general variant switching system in $Q_{T}$:
\begin{eqnarray}\label{generalswitching}
\left\{
\begin{array}{ll}
\displaystyle \max\left\{-\partial_tu_i+\sup_{\alpha\in\mathcal{A}}\inf_{\beta\in\mathcal{B}}\mathcal{L}^{\alpha,\beta}_{i}(t,x,\partial_{x}u_i,\partial_{xx}u_i),u_i-\mathcal{M}^{k}_{i}u\right\}=0, \ \ \ i\in\mathcal{I}:=\{1,...,m\},\\
\displaystyle \max\{u_{m+1}-f,u_{m+1}-\mathcal{M}_{m+1}^{k}u\}=0,
\end{array}
\right.
\end{eqnarray}
with terminal condition
\begin{equation}\label{generalswitching_terminal}
u_i(T,x)=U(x),\ \ \ i\in\bar{\mathcal{I}},
\end{equation} 
where
$$\mathcal{L}^{\alpha,\beta}_{i}(t,x,p,X)=-\frac12\text{Trace}\left(\sigma^{\alpha,\beta}_{i}{\sigma^{\alpha,\beta}_{i}}^T(t,x)X\right)-b^{\alpha,\beta}_{i}(t,x)\cdot
p-L^{\alpha,\beta}_{i}(t,x),$$
$\mathcal{A}$ and $\mathcal{B}$ are compact metric spaces. Note that this variant switching system is a general version of (\ref{switching}), (\ref{Ptbedswitching}), and (\ref{Ptbedswitching1}).
We then make the following assumption:	
\begin{assumption}\label{switchingassumption}
There exists a constant $\bar{C}>0$ independent of $\alpha, \beta$, i and t such that for any $\alpha, \beta$, i, and t, 
$$|\sigma^{\alpha,\beta}_{i}(t,\cdot)|_{1},|b^{\alpha,\beta}_{i}(t,\cdot)|_{1},|L^{\alpha,\beta}_{i}(t,\cdot)|_{1},|f|_{1}, |U|_{1}\leq \bar{C}.$$
\end{assumption}

We firstly give the existence of the solution of (\ref{generalswitching})-(\ref{generalswitching_terminal}), then proceed to the continuous dependence on the coefficients, which implies the standard comparison principle and uniqueness, and finally obtain the regularity for the solution based on continuous dependence results.

\begin{proposition}\label{switchingexistence}
Suppose that Assumption \ref{switchingassumption} is satisfied. Then there exists a bounded viscosity solution $u$ of (\ref{generalswitching})-(\ref{generalswitching_terminal}), with $|u|_{0}\le C$ depending only on $T$ and $\bar{C}$.
\end{proposition}

\begin{proof}
Here we only give an optimal switching representation which can be proved to be a viscosity solution of (\ref{generalswitching})-(\ref{generalswitching_terminal}) using a similar procedure given in \cite{Tang}, and skip the technical proof. Under the same setting as in \cite{Tang}, the viscosity solution of (\ref{generalswitching})-(\ref{generalswitching_terminal}) can be represented on a filtered probability space
$(\Omega,\mathcal{F},\{\mathcal{F}_t\}_{t\geq 0},\mathbf{P})$ by
\begin{align}\label{switchrep}
u_{i}(t,x)= &\ {\inf_{\theta\in\Theta_{i}[t,T],\alpha\in\mathcal{A}[t,T]}}\inf_{\nu\in\mathcal{T}_{i}^{\theta}{[t,T]}}\sup_{\beta\in\mathcal{B}[t,T]}\mathbf{E}^{t,x}\left[\int_{t}^{\tau_\nu}L_{{\theta_{s}}}^{\alpha_{s},\beta_{s}}\left(s,X_{s}^{\alpha,\beta;\theta}\right)ds+k(\nu-\mathbf{1}_{\{\tau_{\nu}= T\}})\right. \nonumber\\
+ &\ \left. f(\tau_\nu,X_{\tau_\nu}^{\alpha,\beta;\theta})\mathbf{1}_{\{\tau_\nu<T\}}+U\left(X_{T}^{\alpha,\beta;\theta}\right)\mathbf{1}_{\{\tau_\nu= T\}}|\mathcal{F}_t\right],
\end{align}
with the state equation
$$dX_{s}^{\alpha,\beta;\theta}=b_{{\theta_{s}}}^{\alpha_{s},\beta_{s}}(s,X_{s}^{\alpha,\beta;\theta})\mathbf{1}_{\{\theta_{s}\neq m+1\}}{ds}+\sigma_{{\theta_{s}}}^{\alpha_{s},\beta_{s}}(s,X_{s}^{\alpha,\beta;\theta})\mathbf{1}_{\{\theta_{s}\neq m+1\}}dW_{s},$$
where $\mathcal{A}[t,T]$ and $\mathcal{B}[t,T]$ are $\mathcal{A}$ and $\mathcal{B}$-valued progressive-measurable processes respectively, and $\Theta_{i}[t,T]$ is the space of all admissible continuous switching control processes on $[t,T]$ starting from $i$. Specifically, for any admissible switching control process, there is a  pair of corresponding sequence $\{\xi_{n}, \tau_{n}\}_{n\ge 0}$ such that $\{\tau_{n}\}_{n\ge 0}$ is a sequence of nondecreasing stopping time  with
$$t=\tau_{0}\leq\tau_{1}\leq\tau_{2}\leq\cdot\cdot\cdot\leq T, \ \ \ \text{a.s.}$$
and that each $\xi_{n}$ is a $\mathcal{F}_{\tau_{n}}$-measurable random variable valued in $\bar{\mathcal{I}}:=\mathcal{I}\cup \{m+1\}$, with $\xi_{0}=i$ and $\xi_{n}\neq\xi_{n+1}$ a.s. Then an admissible switching control process $\theta$ in $\Theta_{i}[t,T]$ is identified as 
$$\theta_{t}=\Sigma_{n\ge 1}\xi_{n-1}\mathbf{1}_{[\tau_{n-1},\tau_{n})}(t).$$
Finally, for each $\theta\in\Theta_{i}[t,T]$, $\mathcal{T}_{i}^{\theta}[t,T]:=\{\nu:\xi_{\nu}=m+1,\tau_{\nu}<T\}\cup\inf\{\nu:\tau_{\nu}=T\}$, which is valued in $\mathbb{N}$ and represents how many times (or minus 1 if $\tau_{\nu}=T$) the agent will switch among different regions before she stops. 

Loosely speaking, the agent in this optimal switching system has to firstly choose a switching control $\theta$ (i.e. a sequence $\{\xi_{n}, \tau_{n}\}_{n\ge 0}$), and then based on it choose how many times to switch among regions before stopping. However, the stop time must be chosen either such that she is in the last region ($i=m+1$), or to be terminal time $T$. 

The representation given by (\ref{switchrep}) immediately shows that $u\ge C$ since \ref{switchingassumption} holds, for some constant $C$ depending only on $T$ and $\bar{C}$. That $u\le C$ can be obtained by simply choosing $\tau_{1}=T$, $\nu=1$ to get rid of the term involving $k$ inside the expectation.
\end{proof}
\begin{remark}
A standard optimal switching system consists only the first $m$ equations in (\ref{generalswitching}). In this case, for each switching control $\theta$, $\mathcal{T}_{i}^{\theta}[t,T]=\inf\{\nu:\tau_{\nu}=T\}$ consists only one element. Representation (\ref{switchrep}) then reduces to a standard optimal switching representation.
\end{remark}

%Note that different from the case in \cite{Tang}, there is an 
The additional equation for $u_{m+1}$ in (\ref{generalswitching}) means that there is an additional region in the variant switching system, and more importantly, the agent in this variant optimal switching problem can not stay in this region: once entering into this region (or be in this region from the start), she only has two options, either stop immediately (to pay a cost of $f(t,x)$) or switch immediately to another region. 

We now give the continuous dependence result for the general variant switching system (\ref{generalswitching}).
\begin{theorem}\label{switchingcd}
For any $s\in(0,T]$, let $u\in USC(\bar{Q}_{s})$ be a
bounded from above viscosity subsolution of (\ref{generalswitching}) with
coefficients $\{\sigma_{i}^{\alpha,\beta},b_{i}^{\alpha,\beta},L_{i}^{\alpha,\beta},f\}$, and
$\bar{u}\in LSC(\bar{Q}_{s})$ be a bounded from below viscosity
supersolution of {(\ref{generalswitching})} with coefficients
$\{\bar{\sigma}_{i}^{\alpha,\beta},\bar{b}_{i}^{\alpha,\beta},\bar{L}_{i}^{\alpha,\beta},\bar{f}\}$. Suppose
that Assumption \ref{switchingassumption} holds for both sets of
coefficients, and that $[u(s,\cdot)]_1\leq M$  or
$[\bar{u}(s,\cdot)]_1\leq M$ for some constant $M$, then there
exists a constant $C$ depending only on $M$, $\bar{C}$, and $s$ such that for each $i\in\bar{\mathcal{I}}$,
\begin{equation}\label{swctdependence}
u_{i}-\bar{u}_{i}\le
C\left(\sup_{i}|(u_{i}(s,\cdot)-\bar{u}_{i}(s,\cdot))^{+}|_{0}+\sup_{i,\alpha,\beta}\left\{|\sigma-\bar{\sigma}|_{0}+|b-\bar{b}|_{0}\right\}+\sup_{i,\alpha,\beta}|L-\bar{L}|_{0}+|f-\bar{f}|_{0}\right)\
\text{in}\ \bar{Q}_s.
\end{equation}
\end{theorem}

\begin{proof}
The proof is mainly based on the proofs of Theorem A.1 in \cite{Jakobsen} and Theorem A.3 in \cite{BJ}, and can be regarded as a combination of these two proofs. We give the details here for reader's convenience.

Fix $0<s\le T$ and in $\bar{Q}_{s}$, we define functions $v(t,x):=e^tu(t,x)$, $\bar{v}(t,x):=e^t\bar{u}(t,x)$ and $g'(t,x):=e^tg(t,x)$ for $g=L, \bar{L}, f, \bar{f}$. It is then follows that $v\in USC(\bar{Q}_{s})$ and $\bar{v}\in LSC(\bar{Q}_{s})$ are bounded above viscosity subsolution and bounded below supersolution of the following system with coefficients $\{\sigma,b, L',f'\}$ and $\{\bar{\sigma},\bar{b}, \bar{L'}, \bar{f'}\}$ respectively:
\begin{eqnarray}
\left\{
\begin{array}{ll}
 \max\left\{-\partial_tv_i+v_{i}+\sup_{\alpha\in\mathcal{A}}\inf_{\beta\in\mathcal{B}}\mathcal{L'}^{\alpha,\beta}_{i}(t,x,\partial_{x}v_i,\partial_{xx}v_i),v_i-\mathcal{M}^{k}_{i}v\right\}=0, \ \ \ i\in\mathcal{I},\\
 \max\{v_{m+1}-f',v_{m+1}-\mathcal{M}_{m+1}^{k}v\}=0,
\end{array}
\right.
\end{eqnarray}
We now use a doubling variables argument to derive a upper bound for $v_{i}-\bar{v}_{i}$ and then will derive for $u_{i}-\bar{u}_{i}$ by using back-substitution.

To continue, we define in $[0,s]\times\mathbb{R}^n\times\mathbb{R}^n$ that $\phi(t,x,y):=e^{\lambda(s-t)}\delta|x-y|^2+\varepsilon(|x|^2+|y|^2)$ and for any $i\in\bar{\mathcal{I}}$, $\psi^{i}(t,x,y):=v_{i}(t,x)-\bar{v}_{i}(t,y)-\phi(t,x,y)$, where $\lambda,\delta,\varepsilon>0$ are positive constants. Then we let $m_{\delta,\varepsilon}^s=\sup_{i,x,y}\psi^{i}(s,x,y)^+$ and $m_{\lambda,\delta,\varepsilon}=\sup_{i,t,x,y}\psi^{i}(t,x,y)-m_{\delta,\varepsilon}^s$. Since $v_{i}$ and $\bar{v}_{i}$ are bounded from above and below respectively, by classical arguments there exists $i_{0}\in\bar{\mathcal{I}}$, $t_0\in[0,s]$ and $x_0,y_0\in\mathbb{R}^n$ depending on $\lambda, \delta$ and $\varepsilon$ such that $\psi^{i_{0}}(t_0,x_0,y_0)=\sup_{i,t,x,y}\psi^{i}(t,x,y)$. Moreover, by Lemma A.2 of \cite{BJ}, $i_{0}$ can be chosen such that $\bar{v}_{i_{0}}(t_{0},y_{0})<\mathcal{M}_{i_{0}}\bar{v}(t_{0},y_{0})$. Loosely speaking this means that we can now ignore the $\bar{v}_{i}-\mathcal{M}_{i}\bar{v}$ parts and proceed as if working with scalar equations. Note that by letting $y=x$, we have for each $i$,
\begin{equation}\label{inequality_1}
m_{\delta,\varepsilon}^s+m_{\lambda,\delta,\varepsilon}=\psi^{i_{0}}(t_0,x_0,y_0)\ge v_{i}(t,x)-\bar{v}_{i}(t,x)-2\varepsilon|x|^2
\end{equation}
for any $(t,x)\in[0,s]\times\mathbb{R}^n$. We now try to derive the upper bound for $m_{\delta,\varepsilon}^s$ and $m_{\lambda,\delta,\varepsilon}$.

Since $[u(s,\cdot)]_1\le M$ or $[\bar{u}(s,\cdot)]_1\le M$, without loss of generality we assume $[\bar{u}(s,\cdot)]_1\le M$, then $[\bar{v}(s,\cdot)]_1\le Me^{s}$ and for any $x,y\in\mathbb{R}^n$
\begin{align*}
\psi^{i}(s,x,y)\le &\ v_{i}(s,x)-\bar{v}_{i}(s,y)-\delta|x-y|^2\\
 \le &\
  \sup_{i}|(v_{i}(s,\cdot)-\bar{v}_{i}(s,\cdot))^{+}|_0+Me^{s}|x-y|-\delta|x-y|^2\\
\le &\
  \sup_{i}|(v_{i}(s,\cdot)-\bar{v}_{i}(s,\cdot))^{+}|_0+M^2e^{2s}\delta^{-1}/4,
\end{align*}
where the last second inequality follows from that $\sup_{r\ge0}(Cr-\delta r^2)= C^2/4\delta$ for any $C,\delta>0$. Thus, we get the upper bound for $m_{\delta,\varepsilon}^s$:
\begin{equation}\label{ub1}
m_{\delta,\varepsilon}^s\le  \sup_{i}|(v_{i}(s,\cdot)-\bar{v}_{i}(s,\cdot))^{+}|_0+C_{1}\delta^{-1},
\end{equation}
where $C_{1}=M^{2}e^{2s}/4.$

On the other hand, we assume that $m_{\lambda,\delta,\varepsilon}>0$ and derive its (positive) upper bound. Of course this upper bound still holds for $m_{\lambda,\delta,\varepsilon}\le0$. Follow this assumption, we have $t_0<s$, since otherwise, $m_{\lambda,\delta,\varepsilon}=\sup_{\mathbb{R}^n\times\mathbb{R}^n}\psi(s,x,y)-m_{\delta,\varepsilon}^s\le0$. Then we can apply the parabolic maximum principle for semicontinuous functions, Theorem 8.3 in \cite{CIL}, to get that there are $a,b\in\mathbb{R}$ and $X,Y\in\mathcal{S}^n$ such that $(a,D_x\phi(t_0,x_0,y_0),X)\in\bar{\mathcal{P}}^{2,+}v_{i_{0}}(t_0,x_0)$ and $(b,-D_y\phi(t_0,x_0,y_0),Y)\in\bar{\mathcal{P}}^{2,-}\bar{v}_{i_{0}}(t_0,y_0)$ with $a-b=\phi_t(t_0,x_0,y_0)$ and the following inequality holds
\begin{equation}\label{ishii}
\left(\begin{array}{cc}
X & 0\\
0 & -Y \end{array} \right) \le 3e^{\lambda(s-t_0)}\delta 
\left(\begin{array}{cc} 
I & -I\\
-I & I \end{array}\right) + 3\varepsilon
\left(\begin{array}{cc} 
I & 0\\
0 & I \end{array}\right).
\end{equation}
We now discuss two different situations where $i_{0}\in\mathcal{I}$ or $i_{0}=m+1$ and derive the upper bounds for $m_{\lambda,\delta,\varepsilon}$ in both situations, then we add them to obtain an upper bound that holds in all situations. 

(i) If $i_{0}\in\mathcal{I}$, then by the definitions of viscosity sub- and supersolutions, as well as the fact that $\bar{v}_{i_{0}}(t_{0},y_{0})<\mathcal{M}_{i_{0}}\bar{v}(t_{0},y_{0})$ and $\sup\inf(\cdot\cdot\cdot)-\sup\inf(\cdot\cdot\cdot)\le\sup\sup(\cdot\cdot\cdot-\cdot\cdot\cdot)$, we have
\begin{align}\label{inequality_2}
0 \le &\ -\lambda e^{\lambda(s-t_{0})}\delta|x_0-y_0|^2-{v}_{i_{0}}(t_0,x_0)+\bar{v}_{i_{0}}(t_0,y_0)\nonumber \\
+ &\
\sup_{\alpha, \beta}\left\{\frac12 \text{tr}[{\sigma_{i_{0}}}\sigma_{i_{0}}^{T}(t_0,x_0)X-\bar{\sigma}_{i_{0}}\bar{\sigma}_{i_{0}}^{T}(t_0,y_0)Y]\right. \nonumber\\
   - &\ 
 \bar{b}_{i_{0}}(t_0,y_0)\cdot(2e^{\lambda(s-t_0)}\delta(x_0-y_0)-2\varepsilon y_0)\nonumber\\
 + &\
 {b}_{i_{0}}(t_0,x_0)\cdot(2e^{\lambda(s-t_0)}\delta(x_0-y_0)+2\varepsilon x_0)\nonumber\\
 - &\
  \left. \bar{L}'_{i_{0}}(t_0,y_{0})+ L'_{i_{0}}(t_0,x_{0})\right\}.
\end{align}
%Now we denote $C$ the constant in Assumption (\ref{switchingassumption}). 
By the inequality (\ref{ishii}) and the fact that $(s+t)^{2}\le 2(s^{2}+t^{2})$ for $s,t\in\mathbb{R}$, we obtain
\begin{align*}
\text{tr}[{\sigma}_{i_{0}}{\sigma}_{i_{0}}^T(t_0,x_0)X-\bar{\sigma}_{i_{0}}\bar{\sigma}_{i_{0}}^T(t_0,y_0)Y] \le &\
 6e^{\lambda(s-t_0)}\delta\left\{|\sigma_{i_{0}}(t_{0},x_{0})-{\bar{\sigma}}_{i_{0}}(t_{0},x_{0})|^{2}+|{\bar{\sigma}}_{i_{0}}(t_{0},x_{0})-{\bar{\sigma}}_{i_{0}}(t_{0},y_{0})|^{2} \right\} \\
 + &\
 3\varepsilon\left\{|{\sigma}_{i_{0}}(t_{0},x_{0})|^{2}+|{\bar{\sigma}}_{i_{0}}(t_{0},y_{0})|^{2} \right\}\\
 \le &\
 6e^{\lambda(s-t_0)}\delta\left\{|\sigma_{i_{0}}-\bar{\sigma}_{i_{0}}|_{0}^{2}+[\bar{\sigma}_{i_{0}}]_{2,1}^{2}|x_{0}-y_{0}|^{2} \right\} \\
 + &\
 3\varepsilon\left\{|\sigma_{i_{0}}|_{0}^{2}+|\bar{\sigma}_{i_{0}}|_{0}^{2} \right\}.
\end{align*}
Furthermore, we have the following estimates
\begin{align*}
(b_{i_{0}}(t_{0},x_{0})-\bar{b}_{i_{0}}(t_{0},y_{0}))\cdot(x_{0}-y_{0}) \le &\ \frac12\left(|b_{i_{0}}(t_{0},x_{0})-\bar{b}_{i_{0}}(t_{0},x_{0})|^{2}+|x_{0}-y_{0}|^{2}\right)\\
+ &\
|\bar{b}_{i_{0}}(t_{0},x_{0})-\bar{b}_{i_{0}}(t_{0},y_{0})||x_0-y_{0}| \\
\le &\
\frac12\left(|b_{i_{0}}-\bar{b}_{i_{0}}|_{0}^{2}+|x_{0}-y_{0}|^{2} \right)+[\bar{b}_{i_{0}}]_{2,1}|x_{0}-y_{0}|^{2},
\end{align*}
\begin{align*}
b_{i_{0}}(t_{0},x_{0})\cdot x_{0} \le &\ |b_{i_{0}}(t_{0},x_{0})||x_{0}| \le ([b_{i_{0}}]_{2,1}|x_{0}|+|b_{i_{0}}(t_{0},0)|)|x_{0}| \\
 \le &\
 \bar{C}(1+|x_{0}|)^{2} \le 2\bar{C}(1+|x_{0}|^{2}),
\end{align*}
and similarly
$$\bar{b}_{i_{0}}(t_{0},y_{0})\cdot y_{0} \le 2\bar{C}(1+|y_{0}|^{2}),$$
$$\bar{v}_{i_{0}}(t_{0},y_{0})-v_{i_{0}}(t_{0},x_{0})\le -\psi(t_{0},x_{0},y_{0})\le -m_{\lambda,\delta,\varepsilon},$$
and
$$ L'_{i_{0}}(t_0,x_{0})-\bar{L}'_{i_{0}}(t_0,y_{0})\le |L'_{i_{0}}-\bar{L}'_{i_{0}}|_{0}+[L'_{i_{0}}]|_{2,1}|x_{0}-y_{0}|.$$
Plugging in all these estimates into inequality (\ref{inequality_2}) yields
\begin{align}\label{ub2}
m_{\lambda,\delta,\varepsilon} \le &\ 3e^{\lambda(s-t_0)}\delta\sup_{i,\alpha,\beta}\left\{|\sigma-\bar{\sigma}|_{0}^{2}+|b-\bar{b}|_{0}^{2} \right\}+e^{s}\sup_{i,\alpha,\beta}|L-\bar{L}|_{0}\nonumber\\
 + &\
  (C_{2}-\lambda)e^{\lambda(s-t_{0})}\delta|x_0-y_0|^2+e^{s}[L_{i_{0}}]|_{2,1}|x_{0}-y_{0}|+
 C_{3}(1+|x_{0}|^{2}+|y_{0}|^{2})\varepsilon.
\end{align}
where $C_{2}=3\bar{C}^{2}+2\bar{C}+1$, $C_{3}=3\bar{C}^{2}+4\bar{C}$ are some positive constants. 

(ii) If $i_{0}=m+1$, then we can easily obtain that
\begin{equation}\label{ub3}
m_{\lambda,\delta,\varepsilon} \le v_{i_{0}}(t_{0},x_{0})-\bar{v}_{i_{0}}(t_{0},y_{0}) \le f'(t_{0},x_{0})-\bar{f}'(t_{0},y_{0})\le e^{s}|f-\bar{f}|_{0}+e^{s}[f]_{2,1}|x_{0}-y_{0}|.
\end{equation}
Combine (\ref{ub2}) and (\ref{ub3}), we get an upper bound for $m_{\lambda,\delta,\varepsilon}$ in all situations. Plug this upper bound as well as (\ref{ub1}) into (\ref{inequality_1}), we have
\begin{align*}
v_{i}(t,x)-\bar{v}_{i}(t,x) \le &\ \sup_{i}|(v_{i}(s,\cdot)-\bar{v}_{i}(s,\cdot))^{+}|_{0}+3e^{\lambda(s-t_0)}\delta\sup_{i,\alpha,\beta}\left\{|\sigma-\bar{\sigma}|_{0}^{2}+|b-\bar{b}|_{0}^{2} \right\}\\
+ &\
 e^{s}\{\sup_{i,\alpha,\beta}|L-\bar{L}|_{0}+|f-\bar{f}|_{0}\}+(C_{2}-\lambda)e^{\lambda(s-t_{0})}\delta|x_0-y_0|^2+2e^{s}\bar{C}|x_{0}-y_{0}|  \\
 + &\ 
C_{1}\delta^{-1}+ C_{3} (1+|x_{0}|^{2}+|y_{0}|^{2})\varepsilon+2\varepsilon |x|^{2}
\end{align*}
Note that this estimate holds for any $\lambda,\delta,\varepsilon>0$. We then try to select appropriate value for them (or take limit) to draw our conclusion. Firstly we may choose $\lambda=C_{2}+1$ and follow again that $\sup_{r\ge0}(Cr-\delta r^2)= C^2/4\delta$ to get rid of the $|x_{0}-y_{0}|$ term. Then, by standard arguments, we know that for any fixed $\lambda$ and $\delta$, $\lim_{\varepsilon\to 0}\varepsilon(|x_{0}|^{2}+|y_{0}|^{2})=0$. By letting $\varepsilon\to 0$, we further get
\begin{align*}
v_{i}(t,x)-\bar{v}_{i}(t,x) \le &\ \sup_{i}|(v_{i}(s,\cdot)-\bar{v}_{i}(s,\cdot))^{+}|_{0}+e^{s}\{\sup_{i,\alpha,\beta}|L-\bar{L}|_{0}+|f-\bar{f}|_{0}\} \\
+ &\
3e^{(C_{2}+1)(s-t_0)}\delta\sup_{i,\alpha,\beta}\left\{|\sigma-\bar{\sigma}|_{0}^{2}+|b-\bar{b}|_{0}^{2}\right\}+(C_{1}+e^{2s}\bar{C}^{2}e^{-(C_{2}+1)(s-t_0)})\delta^{-1}.
 \end{align*}
Note that $\min_{r>0}(ar+br^{-1})=2(ab)^{1/2}$ for any $a,b>0$, we can choose the $\delta$ minimising the right hand side to get
\begin{align*}
v_{i}(t,x)-\bar{v}_{i}(t,x) \le &\  \sup_{i}|(v_{i}(s,\cdot)-\bar{v}_{i}(s,\cdot))^{+}|_{0}+e^{s}\{\sup_{i,\alpha,\beta}|L-\bar{L}|_{0}+|f-\bar{f}|_{0}\} \\
+ &\
 C_{4}\sup_{i,\alpha,\beta}\left\{|\sigma-\bar{\sigma}|_{0}+|b-\bar{b}|_{0}\right\}.
 \end{align*}
where $C_{4}=2(3C_{1}e^{(C_{2}+1)s}+3e^{2s}\bar{C}^{2})^{1/2}$ and we used that $(s^{2}+{t^{2}})^{1/2}\le |s|+|t|$ for any $s,t\in\mathbb{R}$ in the last inequality. Finally, the conclusion follows by back-substituting $v$ and $\bar{v}$ by $u$ and $\bar{u}$.
\end{proof}\\

Finally, by using the above continuous dependence result, we show that the bounded viscosity solution $u$ of (\ref{generalswitching})-(\ref{generalswitching_terminal}) is the unique bounded solution, and moreover, it admits some regularity results.
%Using Proposition \ref{switchingexistence} as well as the continuous dependence result, and following the same arguments as in the proof of Proposition \ref{upperbound_1} in Appendix A, we immediately obtain the uniqueness and regularity of the solution of (\ref{generalswitching})-(\ref{generalswitching_terminal}):
\begin{theorem}\label{switchingthm}
Suppose that Assumption \ref{switchingassumption} is
satisfied. Then, there exists a unique viscosity solution
$u$ of
(\ref{generalswitching})-(\ref{generalswitching_terminal}), with $|u|_{1}\le C$ depending only on $T$ and $\bar{C}$.
\end{theorem}
\begin{proof}
In this proof, we denote by $C$ some constant depending only on $T$ and $\bar{C}$.
Proposition \ref{switchingexistence} gives the existence, and the continuous dependence result (\ref{swctdependence}) implies uniqueness and the $x$-regularity. To proof the $t$-regularity, we follow the idea in Appendix \ref{App1}. 
%Let $\rho(x)$ be a
%$\mathbb{R}_+$-valued smooth function with compact support $B(0,1)$
%and mass $1$, and for $\varepsilon>0$, let $\rho_{\varepsilon}(x):=\frac{1}{\varepsilon^{n}}\rho\left(\frac{x}{\varepsilon}\right)$ be a sequence of mollifiers. 
Fix any $(t,s)$ such that $0\leq t < s \leq T$ and define the functions
$U^{\varepsilon}_{i}:=u_{i}(s,\cdot)*\rho_{\varepsilon}$ in $\mathbb{R}^{n}$ for
$i\in\bar{\mathcal{I}}$ and some $\varepsilon>0$ that shall be decided later, where $\rho_{\varepsilon}$ is the same mollifiers defined in Appendix \ref{App1}. Then similarly we have 
$$|U^{\varepsilon}_{i}-u_{i}(s,\cdot)|_{0}\le C\varepsilon \ \ \text{and} \ \ |D_x^{j}U_{i}^{\varepsilon}|_0\leq C\varepsilon^{1-|j|}.$$ 
Let $u^{\varepsilon}$ be the
unique bounded solution of (\ref{generalswitching}) in $Q_{s}$ with terminal
condition $U^{\varepsilon}$, for
some $\varepsilon>0$ that shall be decided later. The continuous dependence results (\ref{swctdependence}) then implies that for any $i\in\bar{\mathcal{I}}$, 
$$|u_{i}^{\varepsilon}-u_{i}|\le \sup_{i}|U^{\varepsilon}_{i}-u_{i}(s,\cdot)|_{0}\le C\varepsilon \ \ \ \text{in}\ \ \bar{Q}_{s}.$$

Next, for each $i\in\bar{\mathcal{I}}$, define two functions $w^{\pm}_{\varepsilon,i}(t,x):=U_{i}^{\varepsilon}(x)\pm C_{\varepsilon}(s-t)$ in $\bar{Q}_{s}$, for some $C_{\varepsilon}=C(\varepsilon^{-1}+1)$. It then can be easily checked that, for $i\in\mathcal{I}$, the functions  $w^{-}_{\varepsilon,i}$ and $w^{+}_{\varepsilon,i}$ are bounded subsolution and bounded supersolution of
$$-\partial_tw+\sup_{\alpha\in\mathcal{A}}\inf_{\beta\in\mathcal{B}}\mathcal{L}^{\alpha,\beta}_{i}(t,x,\partial_{x}w,\partial_{xx}w)=0, \ \ \text{in} \ Q_{s}.$$
Thus, the function $\bar{v}=(\bar{v}_{1},\bar{v}_{2},...,\bar{v}_{m+1})$ such that $\bar{v}_{i}=w^{+}_{\varepsilon,i}$ for $i\in\mathcal{I}$ and that $\bar{v}_{m+1}=f$, is a bounded supersolution of (\ref{generalswitching}) in $Q_{s}$.  Applying (\ref{swctdependence}) for $u^{\varepsilon}$ and $\bar{v}$ yields that for $i\in\mathcal{I}$,
\begin{align*}
u^{\varepsilon}_{i}(t,x)-w^{+}_{\varepsilon,i}(t,x)
&\ 
 \le 
 \sup_{i}|(U^{\varepsilon}_{i}-\bar{v}_{i}(s,\cdot))^{+}|_{0}=|(U^{\varepsilon}_{m+1}-f(s,\cdot))^{+}|_{0}\\
 &\
 \le
 |(u_{m+1}(s,\cdot)-f(s,\cdot))^{+}|_{0}+C\varepsilon = C\varepsilon,
\end{align*}
which implies that for $i\in\mathcal{I}$,
$$u^{\varepsilon}_{i}(t,x)-U^{\varepsilon}_{i}(x)\le C_{\varepsilon}(s-t)+C\varepsilon.$$
Now we construct a bounded subsolution of (\ref{generalswitching}) in $Q_{s}$ based on $w^{-}_{\varepsilon}$. Note that since $\mathcal{M}_{i}^{k}u$ is concave in $u$, we have for $i\in\bar{\mathcal{I}}$,
$$U^{\varepsilon}_{i}=u_{i}(s,\cdot)*\rho_{\varepsilon}\le \left(\mathcal{M}_{i}^{k}u(s,\cdot)\right)*\rho_{\varepsilon}\le \mathcal{M}_{i}^{k}\left(u(s,\cdot)*\rho_{\varepsilon}\right)=\mathcal{M}_{i}^{k}U^{\varepsilon},$$
and thus $w^{-}_{\varepsilon,i}-\mathcal{M}_{i}^{k}w^{-}_{\varepsilon}\le 0$ in $Q_{s}$. Furthermore, similar to the arguments in Appendix \ref{App1}, we have that in $Q_{s}$, $w^{-}_{\varepsilon,m+1}-f\le C\varepsilon$. Thus, the function $\underline{v}=w^{-}_{\varepsilon}-C\varepsilon$ is a bounded subsolution of (\ref{generalswitching}) in $Q_{s}$.  Applying (\ref{swctdependence}) for  $\underline{v}$ and $u^{\varepsilon}$ yields that for $i\in\mathcal{I}$,
$$w^{-}_{\varepsilon,i}(t,x)-C\varepsilon-u^{\varepsilon}_{i}(t,x)\le 0,$$
which implies that for $i\in\mathcal{I}$,
$$U^{\varepsilon}_{i}(x)-u^{\varepsilon}_{i}(t,x)\le C_{\varepsilon}(s-t)+C\varepsilon.$$

In turn, we obtain that for $i\in\mathcal{I}$,
\begin{align*}
|u_{i}(t,x)-u_{i}(s,x)|\leq &\ |u_{i}(t,x)-u_{i}^{\varepsilon}(t,x)|+|u_{i}^{\varepsilon}(t,x)-U^{\varepsilon}_{i}(x)|+|U^{\varepsilon}_{i}(x)-u_{i}(s,x)|\\
 \leq &\
 2C\varepsilon+C_{\varepsilon}(s-t)+C\varepsilon\\
 \leq &\
 C\left(\varepsilon+\frac{(s-t)}{\varepsilon}+(s-t)\right).
\end{align*}
We choose $\varepsilon=\sqrt{s-t}$ to minimize the right hand
side and obtain that  for $i\in\mathcal{I}$, 
$$|u_{i}(t,x)-u_{i}(s,x)|\leq C\sqrt{s-t}.$$
Note that from (\ref{generalswitching}), we have that $u_{m+1}=\min\{f,u_{1}+k, u_{2}+k,...,u_{m}+k\}$. Therefore, the above inequality also holds for $i=m+1$. This, together with {the boundedness and} the
$x$-regularity, implies that $|u|_{1}\le C$.
\end{proof}

\section{Conclusion}

%In this paper, using the Hopf-Lax formula and the idea of splitting,
%we developed a new numerical scheme for a class of semilinear
%parabolic PDEs with convex and quadratic growth gradients in an
%unbounded domain. It seems the scheme can also be modified so as to
%treat the case of bounded domains. However, it is far more
%challenging to prove the convergence rate of the numerical viscosity
%solutions in bounded domains (see \cite{CS}, \cite{Krylov2} and
%\cite{Reisinger} for some recent developments in this direction). We
%leave such an extension for our future research.

We propose an approximation scheme for a class of semilinear parabolic variational inequalities whose Hamiltonian is convex and coercive. The proposed scheme is a natural extension of a previous splitting scheme proposed by Liang, Zariphopoulou and the author \cite{Huang} for semilinear parabolic PDEs with the same Hamiltonians. 
We establish the convergence of the scheme and determine the convergence rate by obtaining its error bounds. The bounds are obtained by Krylov's shaking coefficients technique and Barles-Jakobsen's optimal switching approximation, while the switching system we use is a variant switching system. Compared to the results in \cite{Huang}, the convergence rates of our proposed scheme remain the same: the upper convergence rate is $1/4$ (Theorem \ref{theorem_error_1}) and the lower convergence rate is $1/10$ (Theorem \ref{theorem_error_2}).

We mention that the approaches and results herein rely heavily on the Lipschitz continuity of (viscosity) solutions of the equation (\ref{PDE_1}) with respect to the space variable. A possible extension is to consider a case when solutions are $\beta$-H\"{o}lder continuous for some $\beta\in(0,1)$. This is challenging in the sense that in this case, (\ref{PDE_1}) can not be written as (\ref{hjbeq111}) where the control set $K$ is compact, which is a key step when obtaining the lower bound. This will be left as future work. One may also consider another version of variational inequalities where the gradient of the solution is constrained rather than the solution itself. These are naturally related to singular stochastic optimization problems. An early result in this direction but in elliptic case can be bound in \cite{Evans2}. 

%We mention that the proposed scheme can be used to approximate   (\ref{PDE_1})-(\ref{terminal}) only when  the unique (viscosity) solution of the equation is Lipschitz continuous with respect to the space variable. Otherwise, (\ref{PDE_1}) can not be written as (\ref{hjbeq111}) where the control set $K$ is compact, which is a key step when obtaining the lower bound. Furthermore, when the proposed scheme is implemented in practice, there is another layer of error coming from the computation of (conditional) expectations (see (\ref{semigroupequation1})). This error is not considered in this paper and in practice, it may influence the convergence rate.
%\section{Weak approximation of the degenerate SDE}
%
%To be done!

%%%%%%%%%%%%%%%%%%%%%%%%%%%%%%%%%%%%%%%%%%%%%%%%%%%%%%%%%%%%%%%55
\begin{appendix}

\section{Proof of Proposition \ref{solutionproperty} and \ref{upperbound_1}}\label{App1}

We note that equation (\ref{PDE_1})-(\ref{terminal}) is a special case (choosing
$\varepsilon=0$) of the equation (\ref{PtbedPDE})-(\ref{Ptbed_terminal}). Therefore, we
omit the proof of Proposition \ref{solutionproperty} and only prove
Proposition \ref{upperbound_1}.

We first show that there exists a bounded solution to (\ref{PtbedPDE})-(\ref{Ptbed_terminal}). To this end, using the convex dual function $L^{\theta}(t,x,q):=\sup_{p\in\mathbb{R}^n}(p\cdot
q-H^{\theta}(t,x,p))$, we rewrite (\ref{PtbedPDE})-(\ref{Ptbed_terminal}) as
\begin{align}
\max\{-\partial_{t}u^{\varepsilon}+{\sup_{\theta\in\Theta^{\varepsilon},q\in\mathbb{R}^{n}}}\mathcal{L}^{\theta,q}\left(t,x,\partial_{x}u^{\varepsilon},\partial_{xx}u^{\varepsilon}\right),u^{\varepsilon}-f(t,x)\}&=0 & \text{in}\ Q_{T+{\varepsilon}^{2}};\label{PtbedPDE11}\\
u^{\varepsilon}(T+{\varepsilon}^{2},x)&=U(x) & \text{in}\ \mathbb{R}^n,\label{Ptbed_terminal11}
\end{align}
where
\begin{equation*}
\mathcal{L}^{\theta,q}(t,x,p,X):=-\frac{1}{2}\text{Trace}\left(\sigma^{\theta}\sigma^{\theta^T}(t,x)X\right)-(b^{\theta}(t,x)-q)\cdot
p-L^{\theta}(t,x,q).
\end{equation*}

We also introduce the stochastic control problem on a filtered probability space
$(\Omega,\mathcal{F},\mathbb{F}=\{\mathcal{F}_t\}_{t\geq 0},\mathbf{P})$:
\begin{align*}
u^{\varepsilon}(t,x)= &\ {\inf_{\theta\in\Theta^{\varepsilon}[t,T+\varepsilon^2],q\in\mathbb{H}^{2}[t,T+\varepsilon^{2}]}}\inf_{\nu\in\mathcal{T}_{[t,T+\varepsilon^{2}]}}\mathbf{E}\left[\int_{t}^{\nu}L^{{\theta_{s}}}\left(s,X_{s}^{t,x;\theta,q},q_{s}\right)ds+f(\nu,X_{\nu}^{t,x;\theta,q})\mathbf{1}_{\{\nu<T+\varepsilon^{2}\}}\right. \\
+ &\ \left. U\left(X_{T+\varepsilon^{2}}^{t,x;\theta,q}\right)\mathbf{1}_{\{\nu=T+\varepsilon^{2}\}}|\mathcal{F}_t\right],
\end{align*}
with the controlled state equation
$$dX_{s}^{t,x;\theta,q}=\left(b^{{\theta_{s}}}(s,X_{s}^{t,x;\theta,q})-q_{s}\right){ds}+\sigma^{{\theta_{s}}}\left(s,X_{s}^{t,x;\theta,q}\right)dW_{s},$$
where $\Theta^{\varepsilon}[t,T+\varepsilon^2]$ is the space of
$\Theta^{\varepsilon}$-valued progressively measurable processes
$(\tau_s,e_s)$, $\mathbb{H}^{2}[t,T+\varepsilon^2]$ is the space of
square-integrable progressively measurable processes $q_s$, for
$s\in[t,T+\varepsilon^2]$, $\mathcal{T}_{[t,T+\varepsilon^{2}]}$ is the collection of all $\mathbb{F}$-stopping times with values in $[t,T+\varepsilon^{2}]$, and $W$ is an d-dimensional Brownian motion with its augmented
filtration $\mathbb{F}$.

Next, we identify its value function with
a bounded viscosity solution to (\ref{PtbedPDE11})-(\ref{Ptbed_terminal11}). For this, we
only need to establish upper and lower bounds for the value function
$u^{\varepsilon}(t,x)$ and, in turn, use standard arguments as in
\cite{Pham} and \cite{Touzi}.

%It is standard to show that $u^{\varepsilon}$ admits the optimal stochastic
%control representation (see, for example, \cite{Pham} and
%\cite{Touzi}):
%\begin{align*}
%u^{\varepsilon}(t,x)= &\ {\inf_{\theta\in\Theta^{\varepsilon}[t,T+\varepsilon^2],q\in\mathbb{H}^{2}[t,T+\varepsilon^{2}]}}\inf_{\nu\in\mathcal{T}_{[t,T+\varepsilon^{2}]}}\mathbf{E}\left[\int_{t}^{\nu}L^{{\theta_{s}}}\left(s,X_{s}^{t,x;\theta,q},q_{s}\right)ds+f(\nu,X_{\nu}^{t,x;\theta,q})\mathbf{1}_{\{\nu<T+\varepsilon^{2}\}}\right. \\
%+ &\ \left. U\left(X_{T+\varepsilon^{2}}^{t,x;\theta,q}\right)\mathbf{1}_{\{\nu=T+\varepsilon^{2}\}}|\mathcal{F}_t\right],
%\end{align*}
%with the state equation
%$$dX_{s}^{t,x;\theta,q}=\left(b^{{\theta_{s}}}(s,X_{s}^{t,x;\theta,q})-q_{s}\right){ds}+\sigma^{{\theta_{s}}}\left(s,X_{s}^{t,x;\theta,q}\right)dW_{s},$$

To find an upper {bound} of $u^{\varepsilon}$,  we choose an arbitrary
perturbation parameter process
$\theta\in\Theta^{\varepsilon}[t,T+\varepsilon]$, an arbitrary stopping time $\nu\in\mathcal{T}_{[t,T+\varepsilon^{2}]}$, and a particular control
$\hat{q}$ for $\hat{q}_{s}\equiv 0$. Then, Proposition 2.3 (ii) in \cite{Huang} yields
\begin{align*}
u^{\varepsilon}(t,x)\le &\ \mathbf{E}\left[\int_{t}^{\nu}L^{{\theta_{s}}}(s,X_{s}^{t,x;\theta,\hat{q}},0)ds+f(\nu,X_{\nu}^{t,x;\theta,\hat{q}})\mathbf{1}_{\{\nu<T+\varepsilon^{2}\}}+U(X_{T+\varepsilon^{2}}^{t,x;\theta,\hat{q}})\mathbf{1}_{\{\nu=T+\varepsilon^{2}\}}|\mathcal{F}_t\right] \\
 \le &\
 (T+\varepsilon^{2}-t)|L^{*}(0)|+M \le C.
\end{align*}
For the lower bound, we use again Proposition 2.3 (ii) in \cite{Huang} to obtain that $L_{*}(q)\ge -H^{*}(0) \ge -|H^{*}(0)|$, for any $q\in\mathbb{R}^{n}$. In turn, for any $(\theta,q,\nu)\in\Theta^{\varepsilon}[t,T+\varepsilon^2]\times\mathbb{H}^2[t,T+\varepsilon^2]\times\mathcal{T}_{[t,T+\varepsilon^{2}]}$,
\begin{align*}
 &\ \mathbf{E}\left[\int_{t}^{\nu}L^{{\theta_{s}}}\left(s,X_{s}^{t,x;\theta,q},q_{s}\right)ds+f(\nu,X_{\nu}^{t,x;\theta,q})\mathbf{1}_{\{\nu<T+\varepsilon^{2}\}}+U\left(X_{T+\varepsilon^{2}}^{t,x;\theta,q}\right)\mathbf{1}_{\{\nu=T+\varepsilon^{2}\}}|\mathcal{F}_t\right] \\
 \ge &\
 \mathbf{E}\left[\int_{t}^{\nu}L_{*}(q_{s})ds|\mathcal{F}_t\right]-M \ge -(T+\varepsilon^{2}-t)|H^{*}(0)|-M \ge C,
\end{align*}
and, thus, $u^{\varepsilon}(t,x)\ge C$ and
$|u^{\varepsilon}|_{0}\le C$. 
%for some constant $C$ independent of
%$\varepsilon$.

The uniqueness of the viscosity solution is a direct consequence of the
following continuous dependence result, whose proof follows along
similar arguments as in Theorem A.1 of \cite{Jakobsen}, and is thus
omitted.

\begin{lemma}\label{Ptbedproperty}
For any $s\in(0,{T+\varepsilon^{2}}]$, let $u\in USC(\bar{Q}_{s})$ be a
bounded from above viscosity subsolution of (\ref{PtbedPDE}) with
coefficients $\{\sigma^\theta,b^\theta\,H^{\theta},f\}$, and
$\bar{u}\in LSC(\bar{Q}_{s})$ be a bounded from below viscosity
supersolution of {(\ref{PtbedPDE})} with coefficients
$\{\bar{\sigma}^\theta,\bar{b}^\theta,\bar{H}^{\theta},\bar{f}\}$. Suppose
that Assumption \ref{data assumption} holds for both sets of
coefficients with respective constants $M$ and $\bar{M}$, uniformly in
$\theta\in\Theta^{\varepsilon}$, and that either $u(s,\cdot)\in\mathcal{C}_b^1(\mathbb{R}^n)$ or
$\bar{u}(s,\cdot)\in\mathcal{C}_b^1(\mathbb{R}^n)$. Then, there
exists a constant $C$, depending only on $M$, $\bar{M}$,
$[u(s,\cdot)]_1$ or $[\bar{u}(s,\cdot)]_1$, and $s$, such that in $\bar{Q}_s$,
\begin{equation}\label{ctdependence}
u-\bar{u}\le
C\left(|(u(s,\cdot)-\bar{u}(s,\cdot))^{+}|_{0}+\sup_{\theta\in\Theta^{\varepsilon}}\left\{|\sigma^\theta-\bar{\sigma}^\theta|_{0}+|b^\theta-\bar{b}^\theta|_{0}\right\}+\sup_{\theta\in\Theta^{\varepsilon}}|H^\theta-\bar{H}^\theta|_{0}+|f-\bar{f}|_{0}\right).
\end{equation}
\end{lemma}

We now continue the proof of Proposition \ref{upperbound_1}. The $x$-regularity of $u^{\varepsilon}$ then follows easily from
(\ref{ctdependence}) by choosing $u=u^{\varepsilon}$,
$\bar{u}=u^{\varepsilon}(\cdot,\cdot+e)$ {and $s=T+\varepsilon^{2}$}.
To get the $t$-regularity, let $\rho(x)$ be a
$\mathbb{R}_+$-valued smooth function with compact support $B(0,1)$
and mass $1$, and for $\varepsilon>0$, let $\rho_{\varepsilon}(x):=\frac{1}{\varepsilon^{n}}\rho\left(\frac{x}{\varepsilon}\right)$ be a sequence of mollifiers.
Next, fix any $(t,s)$ such that $0\leq t < s \leq T+\varepsilon^{2}$ and let $u_{\varepsilon'}$ be the
unique bounded solution of (\ref{PtbedPDE}) in $Q_{s}$ with terminal
condition
$u_{\varepsilon'}(s,x)=u^{\varepsilon}(s,\cdot)*\rho_{\varepsilon'}(x)$, for
some $\varepsilon'>0$ that shall be decided later. It then follows from (\ref{ctdependence}) that,
in $\bar{Q}_{s}$,
$$u^{\varepsilon}-u_{\varepsilon'}\leq C|(u^{\varepsilon}(s,\cdot)-u_{\varepsilon'}(s,\cdot))^{+}|_{0}\leq C[u^{\varepsilon}(s,\cdot)]_{1}\varepsilon'\le C\varepsilon'.$$
Similarly, we also have $u_{\varepsilon'}-u^{\varepsilon}\leq C\varepsilon'$.

On the other hand, standard properties of mollifiers imply that
$|D_x^{j}u_{\varepsilon'}(s,\cdot)|_0\leq C\varepsilon'^{1-|j|}$. Thus, for any $(\xi,x)\in Q_{s}$, we have
$$|\sup_{\theta\in\Theta^{\varepsilon}}g^{\theta}(\xi,x,\partial_{x}u_{\varepsilon'}(s,x),\partial_{xx}u_{\varepsilon'}(s,x))|\le C(\frac{1}{\varepsilon'}+1)=:C_{\varepsilon'}.$$
Define two functions $w^{\pm}_{\varepsilon'}(t,x):=u_{\varepsilon'}(s,x)\pm C_{\varepsilon'}(s-t)$ in $\bar{Q}_{s}$. It then can be easily checked that the function  $w^{+}_{\varepsilon'}(t,x)$
is a bounded supersolution of
(\ref{PtbedPDE}) in $Q_{s}$. Thus, by (\ref{ctdependence}), we have in
$\bar{Q}_{s}$, $u_{\varepsilon'}-w^{+}_{\varepsilon'}\le 0$,
which impies that
$$u_{\varepsilon'}(t,x)-u_{\varepsilon'}(s,x)\leq C_{\varepsilon'}(s-t).$$
We then construct a bounded subsolution of (\ref{PtbedPDE}) in $Q_{s}$ based on $w^{-}_{\varepsilon'}$.
Since $u_{\varepsilon'}(s,x)\le u^{\varepsilon}(s,x)+C\varepsilon'\le f(s,x)+C\varepsilon'$, we obtain that for any $(\xi,x)\in Q_{s}$,
\begin{align*}
w^{-}_{\varepsilon'}(\xi,x)-f(\xi,x) &\ =  u_{\varepsilon'}(s,x)-C_{\varepsilon'}(s-\xi)-f(\xi,x) \\
 &\ 
 \le f(s,x)-f(\xi,x)-C_{\varepsilon'}(s-\xi)+C\varepsilon' \\
 &\ 
 \le M\sqrt{s-\xi}-C_{\varepsilon'}(s-\xi)+C\varepsilon' \\
 &\
 \le \frac{M^{2}}{4C_{\varepsilon'}}+C\varepsilon' \le C\varepsilon',
\end{align*}
where we used $\sup_{r\ge0}(C_{1}r-C_{2} r^2)= C_{1}^2/4C_{2}$ for any $C_{1},C_{2}>0$, and $\frac{1}{C_{\varepsilon'}}\le C\varepsilon'$. This implies that $w^{-}_{\varepsilon'}-C\varepsilon'$ is a bounded subsolution of (\ref{PtbedPDE}) in $Q_{s}$. By (\ref{ctdependence}), we then have in
$\bar{Q}_{s}$, $w^{-}_{\varepsilon'}-C\varepsilon'-u_{\varepsilon'}\le 0$, which impies that
$$u_{\varepsilon'}(s,x)-u_{\varepsilon'}(t,x)\leq C_{\varepsilon'}(s-t)+C\varepsilon'.$$
In turn, we obtain that
\begin{align*}
|u^{\varepsilon}(t,x)-u^{\varepsilon}(s,x)|\leq &\ |u^{\varepsilon}(t,x)-u_{\varepsilon'}(t,x)|+|u_{\varepsilon'}(t,x)-u_{\varepsilon'}(s,x)|+|u_{\varepsilon'}(s,x)-u^{\varepsilon}(s,x)|\\
 \leq &\
 2C\varepsilon'+C_{\varepsilon'}(s-t)+C\varepsilon'\\
 \leq &\
 C\left(\varepsilon'+\frac{(s-t)}{\varepsilon'}+(s-t)\right).
\end{align*}
We choose $\varepsilon'=\sqrt{s-t}$ to minimize the right hand
side and conclude $|u^{\varepsilon}(t,x)-u^{\varepsilon}(s,x)|\leq
C\sqrt{s-t}$, which, together with {the boundedness and} the
$x$-regularity, implies that $|u^{\varepsilon}|_{1}\le C$.

Finally, note that $u(t,x)$ is also the bounded viscosity solution
of (\ref{PtbedPDE}) when $\sigma^\theta\equiv\sigma$,
$b^\theta\equiv b$ and $H^{\theta}\equiv H$. Applying
(\ref{ctdependence}) once more and the regularity of $\sigma$, $b$,
$H$ and $u^\varepsilon$, we deduce that in $\bar{Q}_{T}$,
\begin{align*}
u^\varepsilon-u \le &\ C\left(|(u^\varepsilon(T,\cdot)-u(T,\cdot))^{+}|_0+\sup_{\theta\in\Theta^{\varepsilon}}\left\{|\sigma^\theta-\sigma|_0+|b^\theta-b|_0\right\}+\sup_{\theta\in\Theta^{\varepsilon}}|H^{\theta}-H|_{0}\right)\\
 \le &\
  C\left(|u^\varepsilon(T,\cdot)-u^\varepsilon(T+\varepsilon^2,\cdot)|_0+\varepsilon\right)\le
  C\varepsilon,
\end{align*}
Similarly, we also have $u-u^\varepsilon\leq C\varepsilon$, and we
conclude.

\end{appendix}

%%%%%%%%%%%%%%%%%%%%%%%%%%%%%%%%%%%%%%%%%%%%%%%%%%%55

%%%%%%%%%%%%%%%%%%%%%%%%%%%%%%%

\end{document}